\documentclass{article}

\usepackage[utf8]{inputenc}
\usepackage[normalem]{ulem} 	% allows to break lines when in underline, using \uline rather than \underline
\usepackage{xparse}         	% -> \NewDocumentCommand{}{}{} 
\usepackage{comment}

\usepackage{mathtools,amssymb,amsthm,empheq}
\usepackage{bm} 	   	% allows bold greek symbols, using \bm rather than \boldsymbol
\usepackage{relsize}   	% -> "\mathlarger" that enlarges math symbols
\usepackage{mathrsfs}  	% -> \mathscr, a math font
\usepackage{cancel}   	% -> "\cancel" for "crossing out" diagonals
\usepackage{esint}    	% for the weird integral symbol
\usepackage{xcolor} 	% gives color to cancel
\usepackage{cases}
\usepackage{enumerate}
\usepackage{pgf,tikz}
\usetikzlibrary{cd,arrows,babel,shapes,positioning}
\usepackage{hyperref}
\hypersetup{
	colorlinks,
	citecolor=blue,
	filecolor=blue,
	linkcolor=blue,
	urlcolor=blue
}
\DeclareMathOperator{\Dist}{dist}
\DeclareMathOperator{\Div}{div}

\newcommand{\AP}{A_{\Vert}}
\newcommand{\Z}{\mathbb{Z}}

\newcommand{\supp}{\mathrm{supp}}

\newcommand{\EPS}{\varepsilon}

\newcommand{\ColorWord}[2]{\color{#1} #2 \color{black} }
\numberwithin{equation}{section}

\theoremstyle{plain}

\newtheorem{thm}[equation]{Theorem}
\newcommand{\refthm}[1]{\emph{\ColorWord{blue}{Theorem} \ref{#1}}}

\newtheorem{lemma}[equation]{Lemma}
\newcommand{\reflemma}[1]{\emph{\ColorWord{blue}{Lemma} \ref{#1}}}

\newtheorem{prop}[equation]{Proposition}
\newcommand{\refprop}[1]{\emph{\ColorWord{blue}{Proposition} \ref{#1}}}

\newtheorem{cor}[equation]{Corollary}
\newcommand{\refcor}[1]{\emph{\ColorWord{blue}{Corollary} \ref{#1}}}

\theoremstyle{definition}

\newtheorem{defin}[equation]{Definition}
\newcommand{\refdef}[1]{\emph{Definition \ref{#1}}}

\theoremstyle{remark}

\newtheoremstyle{named}{}{}{\itshape}{}{\bfseries}{}{.5em}{#1 #3}
\theoremstyle{named}
\title{Solvability of the Dirichlet problem using a weaker Carleson condition in the upper half plane}
\author{Martin Ulmer}
\date{\today}

\begin{document}

\maketitle

\begin{abstract}
\noindent
We study an elliptic operator \(L:=\mathrm{div}(A\nabla \cdot)\) on the upper half plane \(\mathbb{R}^2_+\). There are several conditions on the behavior of the matrix \(A\) in the transversal \(t\)-direction that yield \(\omega\in A_\infty(\sigma)\). These include the \(t\)-independence condition, a mixed \(L^1-L^\infty\) condition on \(\partial_t A\), and Dini-type conditions. We introduce an \(L^1\) Carleson condition on \(\partial_t A(x,t)\) that extends the class of elliptic operators for which we have  \(\omega\in A_\infty(\sigma)\), i.e. solvability of the \(L^p\) Dirichlet problem for some \(1<p<\infty\).  
\end{abstract}

\tableofcontents

\section{Introduction}

In this work let \(\Omega:=\mathbb{R}^{2}_+:=\mathbb{R}\times (0,\infty)\) be the upper half space and \(L:=\Div(A\nabla\cdot)\) an uniformly elliptic operator, i.e. there exists \(\lambda_0>0\) such that
\begin{align}
    \lambda_0 |\xi|^2\leq \xi^T A(x,t) \xi \leq \lambda_0^{-1}|\xi|^2 \qquad \textrm{ for all }\xi\in \mathbb{R}^{2},\label{eq:DefinitionOfUniformElliptic}
\end{align}
and a.e. \((x,t)\in \mathbb{R}^{2}_+\). Furthermore assume that the coefficients of \(A\) are bounded and merely measurable. We are interested in the solvability of the \(L^p\) Dirichlet boundary value problem 
\[\begin{cases} Lu=\Div(A\nabla u)=0 &\textrm{in }\Omega, \\ u=f &\textrm{on }\partial\Omega,\end{cases}\]
where \(f\in L^p(\partial\Omega)\) (see \refdef{def:L^pDirichletProblem}). It is well known that solvability for some \(1< p<\infty\) is equivalent to the elliptic measure \(\omega_L\) lying in the Muckenhoupt space \(A_\infty(\sigma)\). This Muckenhoupt space yields a version of scale-invariant absolute continuity between the elliptic measure \(\omega\) and the surface measure \(\sigma\). Due to the counterexamples in \cite{caffarelli_completely_1981} and \cite{modica_construction_1980} we do not always have absolute continuity between \(\omega\) and \(\sigma\), even if we assume that the coefficients of \(A\) are continuous. In fact, these examples show that some regularity on the coefficients in transversal direction is necessary to obtain absolute continuity. This observation gave rise to the study of so called \textit{t-independent} elliptic operators \(L\), i.e. operators where \(A(x,t)=A(x)\) is independent of the transversal direction. The first breakthrough in this direction came in \cite{jerison_dirichlet_1981}, where Jerison and Kenig used a "Rellich" identity to show that if \(A\) is symmetric and t-independent with bounded and measurable coefficients on the unit ball then \(\omega\in B_2(\sigma)\subset A_\infty(\sigma)\). Although this "Rellich" identity could not be transferred to the case of nonsymmetric matrices \(A\), the work \cite{kenig_new_2000} was able to establish \(\omega\in A_\infty(\sigma)\) for nonsymmetric operators in the upper half plane \(\mathbb{R}^2_+\) after about 20 years. Additionally, they give a counterexample why \(\omega\in B_2(\sigma)\) cannot be expected for nonsymmetric matrices and the space \(A_\infty(\sigma)\) is sharp in that case. 

\medskip

For dimension \(n\) however, it took until the Kato conjecture was resolved in \cite{Auscher_Kato} after being open for 50 years, before the authors of \cite{hofmann_square_2014} could extend this result to matrices that are not necessarily symmetric and have merely bounded and measurable coefficients. Later, this work got streamlined in \cite{hofmann_dirichlet_2022}, where the authors also extend the result to the case of matrices whose antisymmetric part might be unbounded and instead has a uniformly bounded BMO norm of its antisymmetric part. The \(t\)-independence condition could also be adapted for the parabolic setting as a sufficient condition for solvability of the \(L^p\) Dirichlet problem (see \cite{auscher_dirichlet_2018}) and was also established for the elliptic regularity problem as sufficient condition for solvability (see \cite{hofmann_regularity_2015}). 

\medskip

A little while after the breakthrough \cite{jerison_dirichlet_1981} the same authors together with Fabes published \cite{fabes_necessary_1984}, where they showed that \(t\)-independence can be relaxed if we use continuous coefficients. More precisely, they show that if a symmetric \(A\) has continuous coefficients, \(\Omega\) is a bounded \(C^1\)-domain, and the modulus of continuity 
\[\eta(s)=\sup_{P\in \partial\Omega, 0<r<s}|A_{ij}(P-rV(P))-A_{ij}(P)|\]
with outer normal vector field \(V\) satisfies the Dini-type condition
\begin{align}\int_0\frac{\eta(s)^2}{s}ds<\infty,\label{DiniTypeCond}\end{align}
then \(\omega\in B_2(\sigma)\subset A_\infty(\sigma)\). Together with the counterexample in \cite{jerison_dirichlet_1981}, where for a given \(\eta\) with \(\int_0 \frac{\eta(s)^2}{s}ds=+\infty\) the authors construct completely singular measures \(\omega\) with respect to the surface measure, the Dini type condition (\ref{DiniTypeCond}) turns out to be sufficient for \(\omega\in A_\infty(d\sigma)\) and kind of necessary, if \(A\) is symmetric with continuous bounded coefficients. Necessity here does not mean that every elliptic operator whose matrix does not satisfy the Dini-type condition \eqref{DiniTypeCond} cannot have an elliptic measure in \(A_\infty(\sigma)\), and hence this is not a classification of solvability. A little bit later, \cite{dahlberg_absolute_1986} extended this result to symmetric matrices with merely bounded and measurable coefficients by considering perturbations from continuous matrices. 

\medskip

The third condition that leads to an elliptic measure in the Muckenhoupt class \(A_\infty(\sigma)\) is an \(L^1-L^\infty\) condition that appears in \cite{ulmer_mixed_2024}. This condition states that if \(|\partial_t A(x,t)|<\frac{c}{t}\) for a.e. \((x,t)\in \mathbb{R}^{n+1}_+\) and
\begin{align}\int_0^\infty\Vert\partial_t A(\cdot,t)\Vert_{L^\infty(\mathbb{R}^n)}dt<\infty,\label{L1-LInftyCond}\end{align}
then \(\omega\in A_\infty(\sigma)\). This condition generalizes the \(t\)-independence condition and hence extends the class of elliptic operators with \(\omega\in A_\infty(\sigma)\). Interestingly, this condition is also different from the Dini-type condition \eqref{DiniTypeCond} with examples that satisfy \eqref{L1-LInftyCond} but not \eqref{DiniTypeCond} and vice versa (cf. \cite{ulmer_mixed_2024}). We would like to point out that the Dini-type condition \eqref{DiniTypeCond} has not yet been shown on the unbounded domain that we are considering here or for non symmetric matrices, and this comparison illustrates that \eqref{L1-LInftyCond} covers indeed a different class of elliptic operators.
\medskip

However, the question of full classification of matrices for which we obtain \(\omega\in A_\infty(\sigma)\) remains open, even for smooth coefficients and symmetric matrices. In trying to make progress towards this question we only look at the case of \(n=1\), i.e. the upper half plane \(\Omega:=\mathbb{R}^{2}_+\). We assume matrices with bounded and merely measurable coefficients which might be non symmetric and introduce the \(L^1\) Carleson condition on \(\partial_t A(x,t)\) as
\begin{align}\sup_{\Delta\subset\partial\Omega \textrm{ boundary ball}}\frac{1}{\sigma(\Delta)}\int_{T(\Delta)}\sup_{(y,s)\in B(x,t,t/2)}|\partial_s A(y,s)| dx dt\leq C<\infty.\label{condition:L^1Carlesontypecond}\end{align}

Now we can state the main theorem of this work.
\begin{thm}\label{thm:MainTheorem}
Let \(L:=\Div(A\nabla\cdot)\) be an elliptic operator satisfying \eqref{eq:DefinitionOfUniformElliptic} on \(\Omega=\mathbb{R}^{2}_+\). If condition \eqref{condition:L^1Carlesontypecond} is satisfied and \(|\partial_t A|t\leq C<\infty\), then \(\omega\in A_\infty(d\sigma)\), i.e. the \(L^p\) Dirichlet boundary value problem for \(L\) is solvable for some \(1<p<\infty\).
\end{thm}

It is clear that on the upper half plane the \(L^1\) Carleson condition on \(\partial_tA\) \eqref{condition:L^1Carlesontypecond} is weaker than the mixed \(L^1-L^\infty\) condition \eqref{L1-LInftyCond}. Hence, this new condition gives rise to a larger class of elliptic operators with \(\omega\in A_\infty(\sigma)\) that is not covered by any other previously mentioned condition.

\medskip
The \(L^1\) Carleson condition on \(\partial_tA\) can also be put into context from a different point of view. In fact, Carleson conditions are a widely applied tool in the area of solvability of boundary value problems in the elliptic and parabolic setting. The famous DKP condition can be stated as 
\begin{align}
    \sup_{\Delta\subset\partial\Omega \textrm{ boundary ball}}\frac{1}{\sigma(\Delta)}\int_{T(\Delta)}|\nabla A(y,s)|^2t dx dt\leq C<\infty,\label{condition:TypicalL^2Carlesontypecond}
\end{align}
if we assume \(|\nabla A(x,t)|\leq \frac{C}{t}\), or with \(\mathrm{osc}_{(y,s)\in B(x,t,t/2)}|A(y,s)|\) instead of the term \(|\nabla A(y,s)|\), and originates from \cite{kenig_dirichlet_2001}. The DKP condition and versions thereof have been studied for elliptic and parabolic boundary value problems and yield solvability of the corresponding boundary value problem in many cases. The elliptic Dirichlet boundary value problem was studied in \cite{kenig_dirichlet_2001}, \cite{dindos_lp_2007}, \cite{hofmann_uniform_2021}, and \cite{david_small_2023}, and also for elliptic operators with complex coefficients (cf. \cite{dindos_regularity_2019}) or elliptic systems (cf. \cite{dindos__2021}). Furthermore, the DKP condition was also successfully studied for the elliptic regularity problem in \cite{dindos_boundary_2017}, \cite{mourgoglou_lp-solvability_2022}, \cite{dindos_etal_regularity_2023}, and \cite{feneuil_alternative_2023}. A helpful survey article of the elliptic setting is \cite{dindos_boundary_2023} which contains further references.
Lastly, we also have positive results for the parabolic Dirichlet and regularity boundary value problem in \cite{dindos_parabolic_2020} and \cite{dindos_lp_2024}. 

\medskip
In contrast to the typical DKP condition \eqref{condition:TypicalL^2Carlesontypecond}, the new Carleson condition in this work \eqref{condition:L^1Carlesontypecond} does not contain any derivative in any other direction than the transversal \(t-\)direction and hence it is easy to see that \eqref{condition:L^1Carlesontypecond} applies to a different class of operators.

\subsection{Idea and sketch of the proof}
According to \cite{dindos_bmo_2011} it is enough to show the local square function bound
\begin{align}\sup_{\Delta\subset\partial\Omega}\sigma(\Delta)^{-1}\int_{T(\Delta)}\vert\nabla u\vert^2\delta dX\lesssim \Vert f\Vert^2_{L^\infty(\partial\Omega)}\label{SquarefctBound}\end{align}
to get \(\omega\in A_\infty(d\sigma)\) for small cubes \(\Delta\).
The first part of the proof follows the outline from \cite{hofmann_dirichlet_2022} or \cite{ulmer_mixed_2024} in which we go back to the idea of a Hodge decomposition of the component \(c\) of \(A\) which originated in \cite{hofmann_square_2015}. However, since \(c\) does depend on \(t\) we have to deal with a family of Hodge decompositions
        \[c(x,t)\chi_{2\Delta}(x)=\AP(x,t)\nabla\phi_t(x) + h_t(x)\qquad \textrm{with \(x\)-divergence free }h_t.\]
and have to control this family uniformly. Since we are only in \(n=1\), we get pointwise bounds of \(\partial_x\phi^s\) and \(\partial_s\partial_x\phi^s\) which we do not find in \cite{ulmer_mixed_2024}. These allow us to find square function bounds of the partial derivatives of the approximating function 
        \[\rho(x,t):= e^{-t^2L^t}\phi_t(x),\]
        where \(L^t:=\Div_x(A(x,t)\nabla_x\cdot)\) is a family of one-dimensional differential operators. This approximating function \(\rho\) plays an important role in the proof as it does in \cite{hofmann_dirichlet_2022} and \cite{ulmer_mixed_2024} and bounding the necessary square function terms of the partial derivatives of \(\rho\) requires the observation in \cite{ulmer_mixed_2024} to decouple the different dependencies of \(t\) by considering 
        \[w(x,t,s):= e^{-t^2L^s}\phi_s(x)\]
        with three variables. This will allow us to obtain a new explicit representation of \(\rho\). The slight improvement of regularity of \(\phi^s\) that we obtain through the restriction to \(n=1\) allows us to eventually replace the mixed \(L^1-L^\infty\) condition \eqref{L1-LInftyCond} from \cite{ulmer_mixed_2024} by something weaker, namely \eqref{condition:L^1Carlesontypecond}. The local nature of the Carleson condition require much more delicate handling of each term. In addition to the tools needed in \cite{ulmer_mixed_2024} like the Kato conjecture and heat semigroup properties, we also make use of off-diagonal estimates for the heat semigroup in delicate ways.
        
        \medskip
        As in \cite{ulmer_mixed_2024} the majority of the work lies in establishing the involved square function bounds in \reflemma{lemma:L^2estimatesForSquareFunctions} and properties of the family \((\nabla\phi^s)_s\), which have to be done differently for our weaker condition \eqref{condition:L^1Carlesontypecond}.

\section*{Acknowledgements}
The author expresses gratitude to Jill Pipher and Martin Dindo\v{s} for many insightful discussions in this topic that helped with refining the applied methods. 
\section{Notations and Setting}

\begin{comment}
Let \(\Omega:=\mathbb{R}^{n+1}_+:=\mathbb{R}^n\times [0,\infty)\) be the upper half space and \(\mathcal{L}:=\Div(A\nabla\cdot)\) an uniformly elliptic operator, i.e. there exists \(\lambda_0>0\)
\begin{align}
    \lambda_0 |\xi|^2\leq \xi^T A(x) \xi \leq \lambda_0^{-1}|\xi|^2 \qquad \textrm{ for all }\xi\in \mathbb{R}^{n+1}\textrm{, and a.e. }x\in \mathbb{R}^{n+1}_+.
\end{align}

\smallskip
\end{comment}

Throughout this work let \(\Omega=\mathbb{R}^{2}_+=\mathbb{R}\times(0,\infty)\) and \(\Delta=\Delta(P,r):=B(P,r)\cap\partial\Omega\) denotes a boundary ball centered in point \(P\in\partial\Omega\) with radius \(l(\Delta)=r>0\) and 
\[T(\Delta):=\{(x,t)\in\Omega; x\in\Delta, 0<t<l(\Delta)\}\]
its Carleson region. 
The cone with vertex in \(P\in \partial\Omega\) and aperture \(\alpha\) is denoted by
\[\Gamma_\alpha(P):=\{(x,t)\in \Omega; |x-P|<\alpha t\}.\] 
Furthermore, we set \(\mathcal{D}^\eta_k(\Delta)\) as the collection of certain boundary balls with radii \(\eta 2^{-k}\) in \(\Delta\) such that their union covers \(\Delta\) but they have finite overlap. %The boundary balls in \(\mathcal{D}^\eta_k(\Delta)\) are a decomposition of \(\Delta\).

\smallskip

We define the nontangential maximal function as
\[N_\alpha(u)(P):=\sup_{(x,t)\in \Gamma_\alpha(P)} |u(x,t)|\qquad \textrm{ for }P\in \mathbb{R}^n,\]
and the mean-valued nontangential maximal function as
\[\tilde{N}_\alpha(u)(P):=\sup_{(x,t)\in \Gamma_\alpha(P)} \Big(\int_{\big\{(y,s); |y-x|<\frac{t}{2}\big\}}|u(y,s)|^2dyds\Big)^{1/2}\qquad \textrm{ for }P\in \mathbb{R}^n.\]

\smallskip

We can now consider the \(L^p\)-Dirichlet boundary value problem for the elliptic operator \(Lu(x,t):=\Div(A(x,t)\nabla u(x,t))\), i.e. for an operator satisfying \eqref{eq:DefinitionOfUniformElliptic}.
\begin{defin}\label{def:L^pDirichletProblem}
We say the \textit{\(L^p\)-Dirichlet boundary value problem} is solvable for \(L\) if for all boundary data \(f\in C_c(\partial\Omega)\cap L^p(\Omega)\) the unique existent solution \(u\in W^{1,2}(\Omega)\) of 
\[\begin{cases} Lu=0 &\textrm{in }\Omega, \\ u=f &\textrm{on }\partial\Omega.\end{cases}\]
satisfies
\[\Vert N(u) \Vert_{L^p(\partial\Omega)}\leq C \Vert f\Vert_{L^p(\partial\Omega)},\]
where the constant \(C\) is independent of \(u\) and \(f\).
\end{defin}
As usual, we denote the elliptic measure by \(\omega\), the usual \(n\)-dimensional Hausdorff surface measure by \(\sigma\) and the Muckenhoupt and Reverse Hölder spaces by \( A_\infty(d\sigma)\) and \(B_q(d\sigma)\) respectively. The solvability of the \(L^p\) Dirichlet boundary value problem is equivalent to \(\omega\in B_{p'}(d\sigma)\), which means that solvability for some \(p\) is equivalent to \(\omega\in A_\infty(d\sigma)\). 

\medskip

In our setting we would like to introduce a bit more notation and specification. In the following we are 
considering the matrix
\begin{align} 
A(x,t)=\begin{pmatrix} \AP(x,t) & b(x,t) \\ c(x,t) & d(x,t)\end{pmatrix}\label{eq:DefinitionOfMatrixA}
\end{align}
for \((x,t)\in \Omega=\mathbb{R}^{n+1}_+\), where \(\AP,b,c,d\) are frunctions from \(\Omega\) to \(\mathbb{R}\). We will also be discussing the family of operators
\[L^sv(x):=\Div_x(\AP(x,s)\nabla v(x)) \qquad\textrm{ for }v\in W^{1,2}(\partial\Omega)=W^{1,2}(\mathbb{R}^n), s>0.\]
At last, we notice that from \(\frac{|\partial_t A|}{t}\leq C\) and the \(L^1\) Carleson condition (\ref{condition:L^1Carlesontypecond}), which is
\begin{align*}\frac{1}{\sigma(\Delta)}\int_{T(\Delta)}\sup_{(y,s)\in B(x,t,t/2)}|\partial_s A(y,s)| dx dt\leq C<\infty,\end{align*}
we have as a consequence the weaker \(L^2\) Carleson type condition
\begin{align}\frac{1}{\sigma(\Delta)}\int_{T(\Delta)}\sup_{(y,s)\in B(x,t,t/2)}|\partial_s A(y,s)|^2s dx dt\leq C<\infty.\label{condition:L^2Carlesontypecond}\end{align}

Since it is sufficient to check \eqref{SquarefctBound} on only small cubes for \(\omega_L\in A_\infty(\sigma)\), we restrict to boundary cubes of size \(l(\Delta)\leq R\) for some fixed \(R>0\).

\section{Consequences of the Kato conjecture and properties of the heat semigroup}

The following two results are the solution to the Kato conjecture which was fully resolved in \cite{Auscher_Kato} for \(p=2\). The \(L^p\) theory for other \(p\) was first fully established in \cite{auscher_necessary_2007} with partial \(L^p\) theory results appearing earlier (please refer to the introduction of \cite{auscher_necessary_2007} for the full history). Note that we define \(\dot W^{1,p}(\mathbb{R}^n)\) as the closure of \(C_0^\infty(\mathbb{R}^n)\) under the seminorm given by \(\Vert f\Vert_{\dot W^{1,p}}:=\Vert \nabla f\Vert_{L^p(\mathbb{R}^n)}\). The Kato \(L^p\) theory for the most general elliptic operator is shown in Thm 4.77 in \cite{hofmann_lp_2022}.
\begin{prop}\label{prop:KatoConjecture}
    For a function \(f\in \dot W^{1,p}(\mathbb{R}^n)\) we have
    \[\Vert L^{1/2} f\Vert_{L^p(\mathbb{R}^n)}\leq C(n,\lambda_0,\Lambda_0)\Vert \nabla f\Vert_{L^p(\mathbb{R}^n)}\]
    for all \(1<p<\infty\). Furthermore, there exists \(\varepsilon_1>0\) such that if \(1<p<2+\varepsilon_1\) then
    \[\Vert \nabla f\Vert_{L^p(\mathbb{R}^n)}\leq C(n,\lambda_0,\Lambda_0)\Vert L^{1/2} f\Vert_{L^p(\mathbb{R}^n)}.\]
\end{prop}

The classical Kato solution is the \(L^2\) case.
\begin{prop}\label{cor:NablaL^-1/2Bounded}
    For a function \(f\in L^2(\mathbb{R}^n)\) we have
    \[\Vert \nabla L^{-1/2} f\Vert_{L^2(\mathbb{R}^n)}\approx\Vert f\Vert_{L^2(\mathbb{R}^n)}\]
    and if f is a vector valued function
    \[\Vert L^{-1/2} \Div(f)\Vert_{L^2(\mathbb{R}^n)}\approx\Vert f\Vert_{L^2(\mathbb{R}^n)}.\]
\end{prop}

The following result assumes a uniform elliptic \(t\)-independent operator \(L\), i.e. \(Lu(x)=\Div_x(\AP(x)\nabla u(x))\) for \(x\in\mathbb{R}^n\).

\begin{prop}[Proposition 4.3  from \cite{hofmann_dirichlet_2022}]\label{prop:L2NormBoundsOfHeatSemigroup}
If \(L\) is t-independent then the following norm estimates hold for \(f\in L^2(\mathbb{R}^n),l\in \mathbb{N}_0, t>0\) and constants \(c_l,C>0\):

\begin{itemize}
    \item \(\Vert \partial_t^l e^{-tL}f\Vert_{L^2(\mathbb{R}^n)}\leq c_l t^{-l} \Vert f\Vert_{L^2(\mathbb{R}^n)}\),
    \item \(\Vert \nabla e^{-tL}f\Vert_{L^2(\mathbb{R}^n)}\leq C t^{-1/2} \Vert f\Vert_{L^2(\mathbb{R}^n)}\).
\end{itemize}
\end{prop}

We have the two following important results
\begin{prop}\label{Proposition11}[Proposition 11 in \cite{hofmann_dirichlet_2022}]
    Let \(\eta>0,\alpha>0\) and L be t-independent. Then we have for all \((y,t)\in\Gamma_{\eta\alpha}(x)\)
    \[\eta^{-1}\partial_t e^{-(\eta t)^2L}f(y)\lesssim M(\nabla f)(x),\]
    and hence
    \[\Vert \eta^{-1} N_{\eta\alpha}[\partial_t e^{-(\eta t)^2L}f]\Vert_{L^p}\lesssim \Vert\nabla f\Vert_{L^p}\]
    for all \(p>1\) and \(f\in W^{1,p}(\mathbb{R}^n)\).
\end{prop}

\begin{prop}[Proposition 12 in \cite{hofmann_dirichlet_2022}]\label{Proposition12}
     Let \(\eta>0\) and \(L\) be t-independent. Then we have 
    \[\Vert \tilde{N}_\eta[\nabla e^{-(\eta t)^2L}f]\Vert_{L^p}\lesssim \Vert\nabla f\Vert_{L^p}\]
    for all \(p>2\) and \(f\in W^{1,p}(\mathbb{R}^n)\).
\end{prop}

We will also apply the following very basic lemma which can be proved with the above introduced statements and ideas.

\begin{lemma}\label{lemma:nablaSemigroupBoundedByNablaf}
Furthermore for t-independent \(L\), if \(f\in W_0^{1,2}(\mathbb{R}^n)\) then 
\[\Vert \nabla e^{-tL}f\Vert_{L^2(\mathbb{R}^n)}\leq C \Vert \nabla f\Vert_{L^2(\mathbb{R}^n)},\]
and
\[\Vert \nabla \partial_t e^{-t^2L}f\Vert_{L^2(\mathbb{R}^n)}\leq C t^{-1}\Vert \nabla f\Vert_{L^2(\mathbb{R}^n)}\]
\end{lemma}

\begin{proof}
    The first inequality is an easy corollary from the Kato conjecture, i.e. we have with \refprop{prop:KatoConjecture} and \refprop{prop:L2NormBoundsOfHeatSemigroup}
    \begin{align*}
        \Vert \nabla e^{-tL}f\Vert_{L^2(\mathbb{R}^n)}&\lesssim \Vert L^{1/2} e^{-tL}f\Vert_{L^2(\mathbb{R}^n)}\lesssim \Vert  e^{-tL}L^{1/2} f\Vert_{L^2(\mathbb{R}^n)}
        \\
        &\lesssim \Vert L^{1/2} f\Vert_{L^2(\mathbb{R}^n)}\lesssim \lesssim \Vert  \nabla f\Vert_{L^2(\mathbb{R}^n)}.
    \end{align*}
    For the second one we have
    \begin{align*}
        \Vert\nabla \partial_t e^{-t^2L}f\Vert_{L^2(\mathbb{R}^n)}^2 &\lesssim \int_{\mathbb{R}^n} A\nabla \partial_t e^{-t^2L}f \cdot \nabla \partial_t e^{-t^2L}f dx
        \\
        &\lesssim \int_{\mathbb{R}^n} \frac{\partial_{tt} e^{-t^2L}f}{2t}\partial_t e^{-t^2L}f dx
        \\
        &\lesssim \frac{1}{t}\Vert\partial_{tt} e^{-t^2L}f\Vert_{L^2(\mathbb{R}^n)} \Vert\partial_t e^{-t^2L}f\Vert_{L^2(\mathbb{R}^n)}
    \end{align*}
    Again, we can observe in the proof of in \refprop{Proposition11} that the only properties of \(\partial_t e^{-t^2L}\) that were used are bounds of its kernel which can be found in Proposition 4.3 of \cite{hofmann_lp_2022}. However, the same Proposition gives the same bounds for \(t\partial_{tt}e^{-t^2L}\) and hence we get 
    \[\Vert\partial_{tt} e^{-t^2L}f\Vert_{L^2(\mathbb{R}^n)}\lesssim \frac{1}{t}\Vert\nabla f\Vert_{L^2(\mathbb{R}^n)}\qquad\textrm{and}\qquad \Vert\partial_{t} e^{-t^2L}f\Vert_{L^2(\mathbb{R}^n)}\lesssim \Vert\nabla f\Vert_{L^2(\mathbb{R}^n)}.\]
    So in total we have
    \[ \Vert\nabla \partial_t e^{-t^2L}f\Vert_{L^2(\mathbb{R}^n)}^2\lesssim \frac{1}{t^2}\Vert\nabla f\Vert_{L^2(\mathbb{R}^n)}^2.\]
\end{proof}

\begin{comment}
Next, we would like to recall \(L^2\) off-diagonal estimates for the heat semigroup. We say that an operator family \(\mathcal{T}_t:L^2\to L^2\) satisfies an \(L^2- L^2\) off-diagonal estimate, if for all closed sets \(E,F\subset \mathbb{R}^n\) and \(h\in L^2(\mathbb{R}^n)\) with support in \(E\)
\[\Vert T_t(h)\Vert_{L^2(F)}\leq C e^{-\frac{c\Dist(E,F)^2}{t}}\Vert h\Vert_{L^2(E)}\]
holds. Note that this definition can be made more general for \(L^p-L^q\) off-diagonal estimates (cf. \cite{hofmann_lp_2022}). However, we just need the \(L^2\) version. 

\begin{prop}[Proposisiton 4.38 in \cite{hofmann_lp_2022}]
    Let L be t-independent.
    The operator family \((\sqrt{t}e^{-tL})_{t>0}\) satisfies \(L^2-L^2\) off-diagonal estimates, i.e.
    \[\Vert e^{-tL}(h)\Vert_{L^2(F)}\leq C \frac{1}{\sqrt{t}}e^{-\frac{c\Dist(E,F)^2}{t}}\Vert h\Vert_{L^2(E)}\]
    for all \(h\in L^2(E)\).
\end{prop}

There are also off-diagonal estimates for \(e^{-tL}\) and \(\partial_te^{-tL}\), and the reader could look them up in \cite{hofmann_lp_2022}.

\end{comment}

Furthermore, we have off-diagonal estimates for operators involving the heat semi group. We quote the following special case of \cite{hofmann_lp_2022}:
\begin{prop}\label{prop:off-diagonal}
    We say an operator family \(T=(T_t)_{t>0}\) satisfies \(L^2-L^2\) off-diagonal estimates, if there exists \(C,\alpha>0\) such that for all closed sets \(E\) and \(F\) and all \(t>0\)
    \[\Vert T_t(h)\Vert_{L^2(F)}\leq Ce^{-\alpha\frac{d(E,F)^2}{t}}\Vert h\Vert_{L^2(E)},\]
    where \(\supp(h)=E\) and \(d(E,F)\) is the semi-distance induced on sets by Euclidean distance.

    \medskip
    Then the families \((e^{-tL})_{t>0}, (t\partial_te^{-tL})_{t>0}\), and \((\sqrt{t}\nabla e^{-tL})_{t>0}\) satisfy \(L^2-L^2\) off-diagonal estimates.
\end{prop}
\section{Hodge decomposition, Approximation function \(\rho\) and difference function \(\theta\)}

\subsection{Hodge decomposition}
For each \(s>0\) we can find a Hodge decomposition consisting of \(\phi^s\in W_0^{1,2}(3\Delta)\) and \( h^s\in L^2(3\Delta)\), where \(h^s\) is divergence free and
\[c(x,s)\chi_{3\Delta}(x)=\AP(x,s)\partial_x\phi^s(x)+h^s(x).\]
Since this is a PDE in one dimension, if we write \(3\Delta=(a, a+3l(\Delta))\), the divergence free function is the constant 
\[h^s=\frac{\fint_{a}^{a+3l(\Delta)} \frac{c(x,s)}{\AP(x,s)}dx}{\fint_{a}^{a+3l(\Delta)} \frac{1}{\AP(x,s)}dx},\]
and 
\[\phi^s(y)=\int_{a}^y \frac{c(x,s)- h^s}{\AP(x,s)} dx\qquad \textrm{ for }x\in 3\Delta. \]

Hence we obtain the following uniform bounds on \(\partial_x \phi^s\) and \(\partial_s\partial_x\phi^s\):
\begin{lemma}\label{lemma:UniformBoundOnM[nablaphi^s]}
There exists \(\kappa_0>0\) such that
\[M[|\partial_x\phi^s|](x)\leq C(\lambda,\Lambda_0)l(\Delta)\leq \kappa_0,\]
and hence
\[\Vert M[\partial_x\phi^s]\Vert_{L^2}\leq C(\lambda,\Lambda_0)l(\Delta)\leq \kappa_0 |\Delta|^{1/2}.\]
Furthermore, we have for every \(x\in 3\Delta\)
\[|\partial_s\partial_x \phi^s(x)|\lesssim |\partial_s A(x,s)| + \fint_{3\Delta}|\partial_s A(y,s)| dy.\]
\end{lemma}

Let \(\eta>0\) be a parameter to be determined later. We define 
\[w(x,t,s):=e^{-tL^s}\phi^s(x)\]
as the solution to the ("t-independent") heat equation
\[\begin{cases} \partial_t w(x,t,s)-L^s w(x,t,s)=0 &,(x,t)\in \Omega \\ w(x,0,s)=\phi^s(x) &,x\in\partial\Omega\end{cases}\]
using the heat semigroup.

We define the by \(\eta\) scaled \textit{approximation function} with ellipticized homogeneity
\[\rho_\eta(x,s):=w(x,\eta^2s^2,s),\]
and the \textit{difference function}
\[\theta_\eta(x,s):=\phi^s(x)-\rho_\eta(x,s).\]

As a consequence of \reflemma{lemma:UniformBoundOnM[nablaphi^s]} we have the following bound on the difference function.

\begin{lemma}\label{lemma:thetaPointwiseBoundOnGoodSetF}
For all \((x,s)\in T(3\Delta)\)
\[\theta_\eta(x,s)\leq C(\lambda_0,\Lambda_0)\kappa_0\eta s.\]
Furthermore, we have any aperture \(\alpha>0\) and
\[\fint_{(x- {\eta\alpha} s, x+{\eta\alpha} s)}|\partial_x\rho_\eta(y,s)|^2 dy \lesssim_{\alpha,\eta}M[\partial_x \phi^s]^2(x)\lesssim \kappa_0^2\]
\end{lemma}

\begin{proof}
    The first statement follows from the maximum principle for parabolic operators. For the second statement we have with a smooth cut-off function \(\psi\in C^\infty\) with \(\mathrm{supp}(\psi)\subset (x-2{\eta\alpha} s, x+2{\eta\alpha} s )\) and \(\psi\equiv 1\) on \((x-{\eta\alpha} s, x+{\eta\alpha} s )\), while \(|\partial_t\psi| + |\partial_x\psi|\lesssim \frac{1}{{\EPS\eta\alpha} s}\). We choose the measn value \(m:=\fint_{(x-2{\eta\alpha} s, x+2{\eta\alpha} s )}\phi^s(x) dx=(\phi^s(\cdot))_{B_x(x, 2\eta\alpha s)}\). Then we get
\begin{align*}
    \fint_{(x-2{\eta\alpha} s, x+2{\eta\alpha} s )}|\partial_x\rho_\eta|^2\psi^2 dy&\leq \fint_{(x-2{\eta\alpha} s, x+2{\eta\alpha} s )}\AP\partial_x\rho_\eta\cdot \partial_x\rho_\eta \psi^2 dy
    \\
    &=\fint_{(x-2{\eta\alpha} s, x+2{\eta\alpha} s ))}\AP\partial_x\rho_\eta\cdot \partial_x((\rho_\eta-m) \psi^2) dy
    \\
    &\qquad + \fint_{(x-2{\eta\alpha} s, x+2{\eta\alpha} s )}\AP\partial_x\rho_\eta\cdot (\rho_\eta-m) \psi\partial_x\psi dy \eqcolon I+J.
\end{align*}
Since \(s L^s \rho_\eta(x,s)=w_t(x,s)\) we have by \refcor{cor:PointwiseBoundsofw_tAndrhoByM(nablaphi^s)}
\[I=\fint_{(x-2{\eta\alpha} s, x+2{\eta\alpha} s )}\partial_t e^{-\eta^2t^2L^s}\phi^s|_{t=s} \frac{(\rho_\eta(y,s)-m)\psi^2}{\eta\alpha s} dy\lesssim M[\partial_x\phi^s](P)^2\]
for every \((P,s)\in \Gamma_{\eta\alpha}(x)\). Choosing \(P\in F\) leads with \reflemma{lemma:UniformBoundOnM[nablaphi^s]} to \(I\lesssim \kappa_0^2\).

\medskip

For \(J\) we can calculate for \(0<\sigma<1\)
\begin{align*}
J&\lesssim \Big(\fint_{(1+\EPS)B_x(x,\eta\alpha s)}|\partial_x\rho_\eta|^2\psi^2 dy\Big)^{1/2}\Big(\fint_{(1+\EPS)B_x(x,\alpha\eta s)}\frac{|\rho_\eta-m|^2}{s^2} dy\Big)^{1/2}
\\
&\lesssim \sigma \fint_{(1+\EPS)B_x(x,\alpha\eta s)}|\partial_x\rho_\eta|^2\psi^2 dy+ C_\sigma M[\partial_x\phi^s](P)^2.
\end{align*}
Hiding the first term on the left side yields
\[\fint_{B_x(x,\alpha\eta s)}|\partial_x\rho_\eta|^2 dy\lesssim \kappa_0^2.\]
\end{proof}

The previous proof used the following result, which is a corollary of \refprop{Proposition11}. We introduce the function \(w_t\) in more detail in the next section when discussing all partial derivatives of \(\theta\).

\begin{cor}\label{cor:PointwiseBoundsofw_tAndrhoByM(nablaphi^s)}
    Let \(\eta>0,\alpha>0\). Then we have for all \((y,t)\in\Gamma_{\eta\alpha}(x)\)
    \begin{align} \eta^{-1}w_t(y,s)=\eta^{-1}\partial_te^{-(\eta t)^2L^s}\phi^s\bigg|_{t=s}(y)\lesssim M(\partial_x \phi^s)(x),\label{eq:w_tpointwiseBounded}\end{align}
    and
    \begin{align}\frac{\rho_\eta(y,s)}{\eta s}=\frac{e^{-(\eta s)^2L^s}(\phi^s-m)(y)}{s}\lesssim M(\partial_x \phi^s)(x),\label{eq:rhopointwiseBounded}\end{align}
    where
    \[m(s):=\fint_{B_x(x, 2\eta\alpha s)}\phi^s(x) dx=(\phi^s(\cdot))_{B_x(x, 2\eta\alpha s)}.\]
\end{cor}

\begin{proof}
    The first inequality \eqref{eq:w_tpointwiseBounded} follows directly from \refprop{Proposition11}. For \eqref{eq:rhopointwiseBounded} we can observe in the proof of \refprop{Proposition11} that the only properties of \(\partial_t e^{-t^2L}\) that were used are bounds of its kernel which can be found in Proposition 4.3 of \cite{hofmann_lp_2022}. However, the same Proposition gives the same bounds for \(\frac{e^{-t^2L}}{t}\) and hence we get with a similar proof to that of \eqref{eq:w_tpointwiseBounded} also \eqref{eq:rhopointwiseBounded}. 
\end{proof}

\subsection{The different parts of the derivatives of \(\theta\)}

We will have to take the partial derivative of \(\theta_\eta\) in the transversal \(t\)-direction. Thereby, we see that
\[\partial_s \theta_\eta(x,s)=\partial_s \phi^s - \partial_s w(x,\eta^2s^2,s)=\partial_s\phi^s-\partial_t w(x,t,s)|_{t=\eta^2 s^2}-\partial_s w(x,t,s)|_{t=\eta^2 s^2}.\]

We can calculate the second and third term more explicitly. For the first of them we get
\begin{align*}
    \partial_t w(x,t,s)|_{t=\eta^2 s^2}=2\eta^2 s L^s w(x,\eta^2 s^2,s).
\end{align*}

For the second one we need to work a bit more. To start with we can observe that
\begin{align*}
    \partial_s\partial_t w(x,t,s)&=\partial_s \partial_x(\AP(x,s)\partial_x w(x,t,s))
    \\
    &=\partial_x(\partial_s \AP(x,s)\partial_x w(x,t,s))+L^s\partial_sw(x,t,s).
\end{align*}
We can now set \(v_1(x,t)\) as the solution to
\[\begin{cases} \partial_t v_1(x,t)=L^s v_1(x,t) &\textrm{for }(x,t)\in\Omega, \\ v_1(x,0)=\partial_s\phi^s(x) &\textrm{for }(x,0)\in\partial\Omega.\end{cases}\]
and \(v_2(x,t)\) as the solution to 
\[\begin{cases} \partial_t v_2(x,t)=L^s v_2(x,t) + \partial_x(\partial_s \AP(x,s)\partial_x w(x,t,s)) &\textrm{for }(x,t)\in\Omega, \\ v_2(x,0)=0 &\textrm{for }(x,0)\in\partial\Omega,\end{cases}\]

Since \(\partial_s w(x,0,s)=\partial_s\phi^s\) we note that \(\partial_s w\) and \(v_1+v_2\) satisfy the same linear PDE with same boundary data and hence must be equal. Next, we can give explicit representations for \(v_2\) by applying Duhamel's principle and for \(v_1\) by the heat semigroup. These are
\[v_1(x,t,s)=e^{-tL^s}\partial_s\phi^s(x),\]
and
\[v_2(x,t,s)=\int_0^t e^{-(t-\tau)L^s}\partial_x(\partial_s A(x,s)\partial_x w(x,\tau,s))d\tau.\]

Together we get for the derivative of \(\theta\) in \(t\)-direction
\begin{align*}
    \partial_s \theta(x,s)&=\partial_s\phi^s-2\eta^2 s L^s w(x,\eta^2 s^2,s)-e^{-\eta^2s^2L^s}\partial_s\phi^s(x)
    \\
    &\qquad-\int_0^{\eta^2s^2} e^{-(\eta^2s^2-\tau)L^s}\partial_x(\partial_s A(x,s)\partial_x w(x,\tau,s))d\tau
    \\
    &=\partial_s\phi^s-2\eta^2 s L^s w(x,\eta^2 s^2,s)-e^{-\eta^2s^2L^s}\partial_s\phi^s(x)
    \\
    &\qquad-\int_0^{s} 2\eta^2\tau e^{-\eta^2(s^2-\tau^2)L^s}\partial_x(\partial_s A(x,s)\partial_x w(x,\eta^2\tau^2,s))d\tau
    \\
    &\eqcolon \partial_s\phi^s(x)-w_t(x,s)-w_s^{(1)}(x,s)-w_s^{(2)}(x,s).
\end{align*}

\begin{lemma}\label{lemma:CaccTypeForw_t}
Let \(\tilde{Q}\subset\Omega\) be a cube and assume also \(4\tilde{Q}\subset\Omega\) and \(s\approx l(\tilde{Q})\). If we set
\[v(x,t,s):=\partial_t w(x,t^2,s)=\partial_t e^{-t^2L^s}\phi_s,\]
we have the Caccioppolli type inequality
\[\int_{\tilde{Q}}|\partial_x v|^2dxdt\lesssim \frac{1}{s^2}\int_{2\tilde{Q}}|v|^2dxdt.\]
\end{lemma}

The proof works analogoulsy to the proof of the other two Caccioppolli type inequalities in \reflemma{Lemma:CacciopolliTypeInequalities}.

\begin{proof}\label{remark:CaccForpartial_tw}
We observe that \(v\) satisfies the PDE
\begin{align*}
\partial_t v(x,t,s)&=\partial_t\partial_t w(x,t^2,s) = \partial_t (-2tL^s w(x,t^2,s)) 
\\
&= -2L^s w(x,t^2,s) - 2tL^s v(x,t,s)=\frac{v(x,t,s)}{t} - 2tL^s v(x,t^2,s).
\end{align*}
The rest of the proof follows standard arguments for Cacciopolli inequalities.
\end{proof}

 To proof \reflemma{lemma:SqFctBoundsForw_s^2} we need two more Cacciopolli type inequalities which are the following

 \begin{lemma}\label{Lemma:CacciopolliTypeInequalities}
     Let \(2\hat{Q}\subset \Omega\) be a Whitney cube, i.e. \(\mathrm{dist}(\hat{Q},\partial\Omega)\approx l(\hat{Q})\approx s\).
     If we set 
     \[v(x,t,s):=v_2(x,t^2,s)=\int_0^t 2\tau e^{-(t^2-\tau^2)L^s}\Div(\partial_s A(x,s)\partial_x e^{-\tau^2 L^s}\phi_s(x))d\tau,\]
     then we have
     \[\int_{\hat{Q}} |\partial_x v(x,t,s)|^2dxdt\lesssim \frac{1}{s^2}\int_{2\hat{Q}}|v(x,t,s)|^2dxdt + \int_{2\hat{Q}} |\partial\AP(x,s)\partial_x e^{-t^2L^s}\phi_s|^2dxdt,\]
     and
     \begin{align*}
         \int_{\hat{Q}} |\partial_x\partial_t v(x,t,s)|^2dxdt&\lesssim \frac{1}{s^2}\int_{2\hat{Q}}|\partial_t v(x,t,s)|^2dxdt
         \\
         & \qquad + \frac{1}{s^2}\int_{2\hat{Q}} |\partial_t\AP(x,s)\partial_x e^{-t^2L^s}\phi_s|^2dxdt
         \\
         &\qquad + \int_{2Q}|\partial_t\AP(x,s)\partial_x\partial_t e^{-t^2L^s}\phi_s|^2dxdt.
     \end{align*}
 \end{lemma}

\begin{proof}
    First we denote \(\tilde{v}:=\partial_t v\) to shorten notation. We observe that \(v\) and \(\tilde{v}\) satisfy the following PDEs
    \begin{align*}
        \partial_t v &= 2t\Div(\partial_s \AP\partial_x e^{-t^2L^s}\phi_s) - \int_0^t 4t\tau L^se^{-(t^2-\tau^2)L^s}\Div(\partial_s \AP(x,s)\partial_x e^{-\tau^2 L^s}\phi_s(x))d\tau
        \\
        &= 2t\Div(\partial_s \AP\partial_x e^{-t^2L^s}\phi_s) - 2t L^s v,
    \end{align*}
    and
    \begin{align*}
        \partial_t \tilde{v} &= \Div(\partial_s \AP\partial_x e^{-t^2L^s}\phi_s) + 2t\Div(\partial_s \AP\partial_x\partial_t e^{-t^2L^s}\phi_s) - 2 L^s v - 2t L^s \tilde{v}.
    \end{align*}

    The rest of the proof is standard for Cacciopolli inequalities of inhomogeneous parabolic PDEs and we omit it here.

\end{proof}

\section{Proof of \refthm{thm:MainTheorem}}

To show that \(\omega\in A_\infty(d\sigma)\) we would like to show \eqref{SquarefctBound}. It is well known that it is enough to show the Carelson measure estimate for only all small boundary balls, i.e. boundary balls \(\Delta\) with \(l(\Delta)\leq R\) for some fixed \(R>0\).
First, we fix a boundary cube \(\Delta\subset\partial\Omega\) and a small \(\eta>0\) that we determine later and assume without loss of generality that \(\Vert f\Vert_{L^\infty(\partial\Omega)}\leq 1\)
Note that the maximum principle implies \(\Vert u\Vert_\infty\leq 1\). To start with we introduce a smooth cut-off of \(T(\Delta)\). Let \(\psi\in C^{\infty}(\mathbb{R}^{n+1}_+)\) with 
\[\psi\equiv 1 \textrm{ on } T(\Delta),\]
and
\[\psi\equiv 0 \textrm{ on } \Omega\setminus (T(3\Delta)).\]
We also have
\[|\nabla\psi(x,t)|\lesssim \frac{1}{l(\Delta)} \qquad \textrm{for all }(x,t)\in \Omega.\]

We can start to estimate the left side of \eqref{SquarefctBound}. Since \(u\Psi^2t\in W^{1,2}_0(\Omega)\) we have
\[\int_{\Omega}A\nabla u\cdot\nabla (u\psi^2t) dxdt=0,\]
and hence
\begin{align*}
    \int_{T(\Delta)}|\nabla u|^2t dxdt&\leq \int_{\Omega}|\nabla u|^2\psi^2t dxdt
    \\
    &\lesssim \int_{\Omega}A\nabla u\cdot\nabla u\psi^2t dxdt
    \\
    &=\int_{\Omega}A\nabla u\cdot\nabla (u\psi^2t) dxdt + \int_{\Omega}A\nabla u\cdot \nabla \psi \psi u t dxdt 
    \\
    &\qquad + \int_{\Omega}A\nabla u\cdot \vec{e}_{2}u\psi^2 dxdt
    \\
    &=\int_{\Omega}A\nabla u\cdot \nabla \psi \psi u t dxdt + \int_{\Omega}A\nabla u\cdot \vec{e}_{2}u\Psi^2 dxdt
    \\
    &\eqcolon J_1+J_2
\end{align*}

For the first term \(J_1\) we have with boundedness of \(A\) that
\begin{align*} 
J_1&\leq \Big(\int_\Omega |\nabla u|^2\psi^2t dxdt\Big)^{1/2}\Big(\int_\Omega |\nabla \psi|^2u^2t dxdt\Big)^{1/2}
\\
&\leq \sigma \int_\Omega |\nabla u|^2\psi^2t dxdt + C_\sigma |\Delta|,
\end{align*}
and for sufficiently small \(\sigma>0\) we can hide the first term on the left side.

For \(J_2\) to continue to break up the integral to
\[\int_{\Omega}A\nabla u\cdot \vec{e}_{2}u\psi^2 dxdt=\int_{\Omega}c\cdot\partial_x u u\psi^2 dxdt+\int_{\Omega}d \partial_t u u\psi^2 dxdt=:J_{21}+J_{22}.\]

We have for \(J_{22}\)
\begin{align*}
    J_{22}&=\int_{\Omega}\partial_t(d u^2\Psi^2) dxdt + \int_{\Omega}d u^2\partial_t\Psi^2 dxdt + \int_{\Omega}\partial_t d u^2\Psi^2
    \\
    &\lesssim \int_{3\Delta}|d u^2\Psi^2| dxdt +|\Delta| + \int_{\Omega}|\partial_t d|
    \\
    &\lesssim |\Delta|.
\end{align*}
Thus it remains to bound \(J_{21}\). 

We proceed with
\begin{align*}
    J_{21}=\int_{\Omega}c\cdot\partial_x (u^2\psi^2) dxdt + \int_{\Omega}c\cdot\partial_x \psi u^2\psi dxdt:=J_{211}+|\Delta|.
\end{align*}

For the term \(J_{211}\) we apply the Hodge decomposition of \(c\) and get
\begin{align*}
    J_{211}&=\int_{\Omega}\AP\partial_x\phi^s \cdot\partial_x(u^2\psi^2)dxdt
    \\
    &=\int_{\Omega}\AP\partial_x \theta_\eta \cdot\partial_x(u^2\psi^2)dxdt - \int_{\Omega}\AP\partial_x\rho \cdot\partial_x(u^2\psi^2)dxdt
    \\
    &=\int_{\Omega}\AP\partial_x \theta_\eta \cdot\partial_x(u^2\psi^2)dxdt  -\int_{\Omega}\AP\partial_x \partial_t\rho \cdot\partial_x(u^2\psi^2)tdxdt  
    \\
    &\qquad   - \int_{\Omega}\AP\partial_x\rho \cdot\partial_x\partial_t(u^2\psi^2)tdxdt - \int_{\Omega}\partial_t\AP\partial_x\rho \cdot\partial_x(u^2\psi^2)tdxdt
    \\
    &=:J_{2111}+I_1+I_2+I_3.
\end{align*}

First we deal with \(I_1,I_2\) and \(I_3\). By \reflemma{lemma:L^2estimatesForSquareFunctions} we have
\[I_1\lesssim \Big(\int_{T(\Delta)}|\partial_x\partial_t\rho|^2tdxdt\Big)^{1/2}\Big(\int_{T(\Delta)}|\partial_x(u^2\psi^2)|^2tdxdt\Big)^{1/2}\lesssim |\Delta|^{1/2} J_1^{1/2}\lesssim |\Delta|.\]

Next, for \(I_2\) we have by \reflemma{lemma:L^2estimatesForSquareFunctions}
\[I_2\lesssim \Big(\int_{T(\Delta)}|\Div(\AP\partial_x\rho)|^2tdxdt\Big)^{1/2}\Big(\int_{T(\Delta)}|\partial_t(u^2\psi^2)|^2tdxdt\Big)^{1/2}\lesssim |\Delta|^{1/2} J_1^{1/2}\lesssim |\Delta|.\]

At last, for \(I_3\) and by \reflemma{lemma:L^2estimatesForSquareFunctions} we get
\[I_3\lesssim \Big(\int_{T(\Delta)}|\partial_t \AP|^2|\partial_x\rho|^2tdxdt\Big)^{1/2}\Big(\int_{T(\Delta)}|\partial_x(u^2\psi^2)|^2tdxdt\Big)^{1/2}\lesssim |\Delta|^{1/2} J_1^{1/2}\lesssim |\Delta|.\]

We continue with \(J_{2111}\) and have
\begin{align*}
    \int_{\Omega}\AP\partial_x \theta_\eta \cdot\partial_x(u^2)\psi^2dxdt + \int_{\Omega}\AP\partial_x \theta_\eta \cdot u^2\partial_x(\psi^2)dxdt=:II_1+II_2.
\end{align*}

For \(II_2\) we observe that
\begin{align*} 
    II_2&\lesssim \Big(\int_{T(3\Delta)}\frac{|\partial_x\theta_\eta|^2}{l(\Delta)^2}dxdt\Big)^{1/2}|\Delta|^{1/2}\lesssim |\Delta|.
\end{align*}

For \(II_1\) we get
\begin{align*}
    II_1&=\int_{\Omega}\AP\partial_x u \cdot\partial_x(\theta_\eta u \psi^2)dxdt + \int_{\Omega}\AP\partial_x u \cdot\partial_x u \theta_\eta \psi^2dxdt + \int_{\Omega}\AP\partial_x u^2 \cdot\partial_x \psi^2 \theta_\eta dxdt
    \\
    &=:II_{11}+II_{12}+II_{13}.
\end{align*}

Due to \reflemma{lemma:thetaPointwiseBoundOnGoodSetF} we have that \(II_{13}\) can be handled like \(J_1\) and 
\[II_{12}=\eta \kappa_0 \int_{T(\Delta)}|\partial_x u|^2\psi^2 tdxdt,\]
and choosing \(\eta\) sufficiently small let us hide this term on the left side. To deal with \(II_{11}\) we use that \(u\) is a weak solution and \(\theta_\eta u \psi^2\in W_0^{1,2}(\Omega)\) to conclude
\begin{align*}
    II_{11}&=\int_{\Omega}\AP\partial_x u \cdot\partial_x(\theta_\eta u \psi^2)dxdt
    \\
    &= \int_{\Omega}b\partial_t u \cdot\partial_x(\theta_\eta u \psi^2)dxdt
    + \int_{\Omega}c\cdot \partial_x u \partial_t(\theta_\eta u \psi^2)dxdt
    + \int_{\Omega}d\partial_t u \partial_t(\theta_\eta u \psi^2)dxdt
    \\
    &=:II_{111}+II_{112}+II_{113}.
\end{align*}

We have with \reflemma{lemma:thetaPointwiseBoundOnGoodSetF}
\begin{align*}
    II_{113}&=\int_{\Omega}d\partial_t u \partial_t\theta_\eta u \psi^2 dxdt + \int_{\Omega}d\partial_t u \theta_\eta \partial_t u \psi^2 dxdt + \int_{\Omega}d\partial_t u \theta_\eta u \partial_t \psi^2 dxdt
    \\
    &\lesssim \Big(\int_{T(3\Delta)}\frac{|\partial_t\theta_\eta|^2}{t} dxdt\Big)^{1/2}\Big(\int_{T(\Delta)}|\partial_t u|^2\psi^2 tdxdt\Big)^{1/2} 
    + \int_{T(\Delta)}|\partial_t u|^2\psi^2 \theta_\eta dxdt
    \\
    &\qquad + \int_{\Omega}|\partial_t u \theta_\eta u \partial_t \psi^2| dxdt
    \\
    &\lesssim C_\sigma \int_{T(3\Delta)}\frac{|\partial_t\theta_\eta|^2}{t} dxdt + \sigma \int_{T(\Delta)}|\partial_t u|^2\psi^2 tdxdt
    \\
    &\qquad +\eta \int_{T(\Delta)}|\partial_t u|^2\psi^2 t dxdt +  \int_{\Omega}|\partial_t u^2 \partial_t \psi^2|t dxdt.
\end{align*}

Hence, for a sufficiently small choice of \(\eta\) and \(\sigma\) we can hide the second and third term on the left side, while the forth term is bounded like \(J_1\) and the first one with \reflemma{lemma:L^2estimatesForSquareFunctions}.

Analogoulsy, we get for \(II_{112}\)
\begin{align*}
    II_{113}&=\int_{\Omega}c\cdot\partial_x u \partial_t\theta_\eta u \psi^2 dxdt + \int_{\Omega}c\cdot\partial_x u \theta_\eta \partial_t u \psi^2 dxdt + \int_{\Omega}c\cdot\partial_x u \theta_\eta u \partial_t \psi^2 dxdt
    \\
    &\lesssim C_\sigma \int_{T(3\Delta)}\frac{|\partial_t\theta_\eta|^2}{t} dxdt + \sigma \int_{T(\Delta)}|\partial_x u|^2\psi^2 tdxdt
    \\
    &\qquad +\eta \int_{T(\Delta)}|\partial_x u|^2\psi^2 t dxdt +  \int_{\Omega}|\partial_x u^2 \partial_t \psi^2|t dxdt,
\end{align*}

and hence \(II_{113},II_{112}\lesssim |\Delta|\). For \(II_{111}\) however, we write
\begin{align*}
    II_{111}=\int_{\Omega}b\cdot \partial_x\theta_\eta \partial_t u^2 \psi^2 dxdt + \int_{\Omega}b\cdot \partial_x u \partial_t u^2 \psi^2 \theta_\eta dxdt + \int_{\Omega}b\cdot \partial_x \psi u\partial_t u\psi \theta_\eta dxdt
    \\
    =:II_{1111}+II_{1112}+II_{1113}.
\end{align*}
Since 
\[II_{1112}\lesssim \eta \int_{T(\Delta)}|\partial_x u|^2\psi^2 t dxdt,\]
and
\[II_{1113}\lesssim \Big(\int_{T(3\Delta)}\frac{|u|^2}{l(\Delta)^2}t dxdt\Big)^{1/2}\Big(\int_{T(3\Delta)}|\partial_t u|^2\psi^2 tdxdt\Big)^{1/2}\]
we get boundedness of these two terms like before.

Next, we recall that \(\partial_t\theta_\eta=\partial_t\phi^t-(w_t+w_s^{(1)}+w_s^{(2)})\) and get by integration by parts
\begin{align*}
     II_{1111}&=\int_{\Omega}b\cdot \partial_x\partial_t\theta_\eta u^2 \psi^2 dxdt + \int_{\Omega}b\cdot \partial_x \theta_\eta u^2\psi \partial_t\psi  dxdt + \int_{\Omega}\partial_t b\cdot \partial_x \theta_\eta u^2\psi^2  dxdt
    \\
    &\lesssim \int_{\Omega}b\cdot \partial_x(w_t u^2 \psi^2) dxdt + \int_{\Omega}b\cdot \partial_x(\partial_t\phi^t-w_s^{(1)}+w_s^{(2)}) u^2 \psi^2 dxdt
    \\
    &\qquad +\int_{\Omega}b\cdot \partial_x (u^2 \psi^2) w_t dxdt
    + \int_{\Omega}b\cdot \partial_x \theta_\eta u^2\psi \partial_t\psi  dxdt + \int_{\Omega}\partial_t b\cdot \partial_x \theta_\eta u^2\psi^2  dxdt
    \\
    &=:III_{1}+III_{2}+III_{3}+III_{4}+III_5.
\end{align*}
We can bound \(III_{4}\) like \(II_2\), while with Cauchy-Schwarz, \reflemma{lemma:SqFctBoundsForw_t} and with hiding of terms on the left side we have
\begin{align*}
    III_{3}&\lesssim \int_{T(3\Delta)}|w_t\partial_x u u\psi^2|+|w_t u^2\partial_x\psi \psi|dxdt
    \\
    &\lesssim \Big(\int_{T(3\Delta)}\frac{|w_t|^2}{t} dxdt\Big)^{1/2}\Big(\int_{T(3\Delta)}|\partial_x u|^2\psi^2 t dxdt\Big)^{1/2}
    \\
    &\qquad + \Big(\int_{T(3\Delta)}\frac{|w_t|^2}{t} dxdt\Big)^{1/2}\Big(\int_{T(3\Delta)}|u\partial_x\psi|^2 t dxdt\Big)^{1/2}.
\end{align*} 
We obtain boundedness by \(|\Delta|\) for \(III_3\). 

For \(III_2\) we observe that due to \eqref{lemma:w_s^2SqFctBound3} in \reflemma{lemma:SqFctBoundsForw_s^2}, and \eqref{lemma:w_s^1SqFctBound3} in \reflemma{lemma:SqFctBoundsForpartial_sPhiAndw_s^1} we have
\begin{align*}
    III_2&\lesssim\int_{T(3\Delta)}|\partial_x\partial_t\phi^t| +|\partial_x w_s^{(1)}|+|\partial_x w_s^{(2)}| dxdt\lesssim |\Delta|
\end{align*}

For \(III_5\) we get with \reflemma{lemma:UniformBoundOnM[nablaphi^s]} and \reflemma{lemma:thetaPointwiseBoundOnGoodSetF}
\[\int_{\Omega}|\partial_t b\cdot \partial_x \theta_\eta u^2\psi^2|  dxdt\lesssim \int_{T(3\Delta)}\sup_{B(x,t,t/2)}|\partial_s b(y,s)|dxdt\lesssim |\Delta|.\]

For \(III_1\) however, we need the Hodge decomposition of \(b\). Since \(\partial\theta_\eta \psi^2 u^2\in W_0^{1,2}(T(3\Delta))\) we have

\begin{align*}
    III_1= \int_{\Omega}\AP\partial_x\tilde{\theta}\cdot \partial_x(w_t u^2 \psi^2) dxdt + \int_{\Omega}\AP\partial_x\tilde{\rho}\cdot \partial_x(w_t u^2 \psi^2) dxdt
    =:III_{11}+III_{12}.
\end{align*}

With integration by parts the second term becomes
\[III_{12}\lesssim \Big(\int_{T(3\Delta)}\frac{|w_t|^2}{t} dxdt\Big)^{1/2}\Big(\int_{T(3\Delta)}|L^t\tilde{\rho}|^2 t dxdt\Big)^{1/2}\lesssim |\Delta|,\]
where the last inequality follows from \reflemma{lemma:L^2estimatesForSquareFunctions} and \reflemma{lemma:SqFctBoundsForw_t}.

For \(III_{11}\) we obtain
\begin{align*}
    III_{11}&=\int_{\Omega}\AP\partial_x\tilde{\theta}\cdot \partial_x u^2 w_t \psi^2 dxdt + \int_{\Omega}\AP\partial_x\tilde{\theta}\cdot \partial_x \psi^2 w_t u^2  dxdt + \int_{\Omega}\AP\partial_x\tilde{\theta}\cdot \partial_x w_t u^2 \psi^2 dxdt
    \\
    &=:III_{111}+III_{112}+III_{113}.
\end{align*}

First, we see that 
\[III_{112}=\Big(\int_{T(3\Delta)}\frac{|w_t|^2}{t} dxdt\Big)^{1/2}\Big(\int_{T(3\Delta)}|\partial_x\tilde{\theta}|^2|\partial_x\psi|^2 t dxdt\Big)^{1/2}\] 
can be dealt with by \(II_2\) and \reflemma{lemma:SqFctBoundsForw_t}.

Next, for \(III_{113}\) we obtain by integration by parts
\begin{align*}
    III_{113}&=\int_{\Omega} L^t w_t u^2\psi^2 \tilde{\theta} dxdt + \int_{\Omega} \AP \partial_x w_t \partial_x(u^2\psi^2) \tilde{\theta} dxdt,
\end{align*}
where the first term can be bounded by
\[\Big(\int_{T(3\Delta)}\frac{|\tilde{\theta}|^2}{t^3} dxdt\Big)^{1/2}\Big(\int_{T(3\Delta)}|L^tw_t|^2 t^3 dxdt\Big)^{1/2}\lesssim |\Delta|\]
due to \reflemma{lemma:L^2estimatesForSquareFunctions} and \reflemma{lemma:SqFctBoundsForw_t}, while the second term is bounded by
\begin{align*}
    &\int_{T(3\Delta)}|\partial_x w_t\partial_x u u\psi^2|t+|\partial_x w_t u^2\partial_x\psi \psi|tdxdt
    \\
    &\lesssim \Big(\int_{T(3\Delta)}|\partial_x w_t|^2t dxdt\Big)^{1/2}\Big(\int_{T(3\Delta)}|\partial_x u|^2\psi^2 t dxdt\Big)^{1/2}
    \\
    &\qquad + \Big(\int_{T(3\Delta)}|\partial_x w_t|^2 t dxdt\Big)^{1/2}\Big(\int_{T(3\Delta)}|u\partial_x\psi|^2 t dxdt\Big)^{1/2},
\end{align*}
and hence by hiding terms on the left side and \reflemma{lemma:SqFctBoundsForw_t}.

At last, it remains to bound \(III_{111}\). For that we have with \reflemma{lemma:thetaPointwiseBoundOnGoodSetF}
\[III_{111}\lesssim \sigma\int_{T(3\Delta)}|\partial_x u|^2\psi^2 t dxdt + C_\sigma\int_{T(3\Delta)}\frac{|w_t|^2}{t} dxdt.\]
The first term can be hiden on the left side with sufficiently small choice of \(\sigma\). For the second one we write 
\begin{align*}
    \int_{T(3\Delta)}\frac{|\partial_x\tilde{\theta}|^2|w_t|^2}{t} dxdt&
    =\sum_{k\leq k_0}\sum_{Q^\prime\in \mathcal{D}^\eta_k(5\Delta)}\int_{Q^\prime}\int_{2^{-k}}^{2^{-k+1}}\frac{|\partial_x\tilde{\theta}|^2|w_t|^2}{t}dxdt
    \\
    &=\sum_{k\leq k_0}\sum_{Q^\prime\in \mathcal{D}^\eta_k(5\Delta)}\Big(\fint_{Q^\prime}\int_{2^{-k}}^{2^{-k+1}}\frac{|\partial_x\tilde{\theta}|^2}{t}dxdt\Big)\Big(|Q|\sup_{Q\times (2^{-k},2^{-k+1}]}|w_t|\Big)
\end{align*}
Using \reflemma{lemma:UniformBoundOnM[nablaphi^s]} and \reflemma{lemma:UniformBoundOnM[nablaphi^s]} we see that
\[\fint_{Q^\prime}\fint_{2^{-k}}^{2^{-k+1}}|\partial_x\tilde{\theta}|^2dxdt=\fint_{Q^\prime}\fint_{2^{-k}}^{2^{-k+1}}|\partial_x\tilde{\rho}|^2dxdt+\fint_{Q^\prime}\fint_{2^{-k}}^{2^{-k+1}}|\partial_x\tilde{\phi}^t|^2dxdt\lesssim \kappa_0,\]

and hence with \reflemma{lemma:localHarnackTypeInequalityForw_t}
\begin{align*}
    \int_{T(3\Delta)}\frac{|\partial_x\tilde{\theta}|^2|w_t|^2}{t} dxdt
    &\lesssim\sum_{k\leq k_0}\sum_{Q^\prime\in \mathcal{D}^\eta_k(5\Delta)}\int_{Q^\prime}\int_{2^{-k}}^{2^{-k+1}}\frac{|w_t|^2}{t}dxdt\lesssim |\Delta|.
\end{align*}
The last inequality follows again from \reflemma{lemma:L^2estimatesForSquareFunctions}. In total we verified the Carleson measure estimate \eqref{SquarefctBound}.

\section{Square function bounds on the parts of the derivative of \(\theta\)}

We investigate square function type expression involving of \(\partial_s\phi^s,w_t,w_s^{(1)}\) and \(w_s^{(2)}\). For all of the results in this chapter we assume the \(L^1\) Carleson condition \eqref{condition:L^1Carlesontypecond}.

\subsection{The partial derivative parts \(\partial_s\phi^s\) and \(w_s^{(1)}\)}

Instead of investigating \(\partial_s\phi^s\) and \(w_s^{(1)}\) separately, we are also going to use that their difference
\[\partial_s\phi^s-w_s^{(1)}=\partial_s\phi^s-e^{-(\eta s)^2L^s}\partial_s\phi^s\]
gives us better properties.

\begin{lemma}\label{lemma:SqFctBoundsForpartial_sPhiAndw_s^1}
The following square function bounds hold
\begin{enumerate}[(i)]
    \item \[\int_{T(3\Delta)}\frac{|\partial_s\phi^s-w_s^{(1)}|^2}{s}dxds\lesssim |\Delta|,\]\label{lemma:w_s^1SqFctBound1}
    \item \[\int_{T(3\Delta)}\frac{|s\partial_x\partial_s\phi^s|^2}{s}dxds+\int_{T(3\Delta)}\frac{|s\partial_x w_s^{(1)}|^2}{s}dxds\lesssim |\Delta|.\label{lemma:w_s^1SqFctBound2}\]
    \item \[\int_{T(3\Delta)}|\partial_x\partial_s\phi^s|dxds + \int_{T(3\Delta)}|\partial_x w_s^{(1)}|dxds\lesssim |\Delta|.\label{lemma:w_s^1SqFctBound3}\]
\end{enumerate}
\end{lemma}

\begin{proof}
    For \eqref{lemma:w_s^1SqFctBound1} we use an analogous way to the proof of \reflemma{lemma:thetaPointwiseBoundOnGoodSetF}, where we use \refprop{Proposition11} to get
    \begin{align*}
        (\partial_s\phi^s-e^{-(\eta s)^2L^s}\partial_s\phi^s)(x)&=\int_0^s\partial_t e^{-(\eta t)^2L^s}\partial_s \phi^sdt\lesssim \eta\int_0^s M[\partial_x\partial_s\phi^s](x)dt
        \\
        &\lesssim M[\partial_x\partial_s\phi^s](x)\eta s.
    \end{align*}
    Hence by \reflemma{lemma:UniformBoundOnM[nablaphi^s]},
    \begin{align*}
       \int_{T(3\Delta)}\frac{|\partial_s\phi^s-w_s^{(1)}|^2}{s}dxds &\lesssim \int_0^{l(3\Delta)} \int_{3 \Delta}|M[\partial_x\partial_s\phi^s](x)|^2sdxds
        \\
        &\lesssim \int_0^{l(3\Delta)} \int_{3\Delta}|\partial_x\partial_s\phi^s(x)|^2sdxds
        \\
        &\lesssim \int_0^{l(3\Delta)} \int_{3\Delta}|\partial_s A|^2sdxds
        \\
        &\lesssim |\Delta|,
    \end{align*}
    
    and observation \eqref{condition:L^2Carlesontypecond} finishes the proof of \eqref{lemma:w_s^1SqFctBound1}.

    \medskip
    For \eqref{lemma:w_s^1SqFctBound2} we just observe that \(\Vert\partial_x w_s^{(1)}\Vert_{L^2(\mathbb{R}^n)}\lesssim \Vert \partial_x\partial_s\phi^s\Vert_{L^2(\mathbb{R}^n)}\), which implies
    \begin{align*}
        \int_{T(3\Delta)}\frac{|s\partial_x\partial_s\phi^s|^2}{s}dxds&,\int_{T(3\Delta)}\frac{|s\partial_x w_s^{(1)}|^2}{s}dxds
        \\
        &\qquad\qquad\lesssim \int_0^{l(3\Delta)} \int_{3\Delta}|\partial_x\partial_s\phi^s(x)|^2sdxds.
    \end{align*}
    Following the argument for \eqref{lemma:w_s^1SqFctBound1} from here finishes the proof for \eqref{lemma:w_s^1SqFctBound2}.

    \medskip
    Lastly, we deal with \eqref{lemma:w_s^1SqFctBound3}. By \reflemma{lemma:UniformBoundOnM[nablaphi^s]} and \reflemma{lemma:thetaPointwiseBoundOnGoodSetF} we have
    \[\int_{T(3\Delta)}|\partial_x\partial_s\phi^s|dxds\lesssim \int_{T(3\Delta)}|\partial_s\AP|dxds\lesssim |\Delta|,\]
    and
    \begin{align*}
        \int_{T(3\Delta)}|\partial_x w_s^{(1)}|dxds&=\int_0^{3l(\Delta)}\sum_{\substack{Q\subset 5\Delta,\\l(Q)\approx s, Q\textrm{ finite overlap}, \bigcup Q\supset 3\Delta}}\Big(\fint_Q |\partial_x e^{-\eta^2s^2L^s\partial_s\phi^s}|^2 dx\Big)^{1/2}|Q|
        \\
        &\lesssim \int_0^{3l(\Delta)}\sum_{\substack{Q\subset 5\Delta,\\l(Q)\approx s, Q\textrm{ finite overlap}, \bigcup Q\supset 3\Delta}}\sup_{B(x,s,s/2)}|\partial_s\AP||Q|
        \\
        &\lesssim \int_{T(5\Delta)}\sup_{B(x,s,s/2)}|\partial_s\AP|\lesssim |\Delta|.
    \end{align*}
\end{proof}

\subsection{The partial derivative part \(w_s^{(2)}\)}

\begin{lemma}\label{lemma:w_s^2/nablaw_s^2spatialL^2Bound}
    It holds that
\begin{align}
\int_{T(3\Delta)}\frac{|w_s^{(2)}|^2}{s}dxds\lesssim |\Delta|.\label{lemma:w_s^2SqFctBound1}
\end{align}
\end{lemma}

\begin{proof}
    We use Minkowski inequality and \refprop{prop:KatoConjecture} to obtain

    \begin{align*}
        &\int_{T(3\Delta)}\frac{|w_s^{(2)}|^2}{s}dxds= \int_0^{3l(\Delta)}\int_{3\Delta}\frac{1}{s}\Big|\int_0^{s} \eta^2\tau e^{-\eta^2(s^2-\tau^2)L^s}\partial_x(\partial_s \AP(x,s)\partial_x w(x,\eta^2\tau^2,s))d\tau\Big|^2 dx ds
        \\
        &\qquad\lesssim \eta^2\int_0^{3l(\Delta)}\frac{1}{s}\Big(\int_0^{s}\tau \Vert (L^s)^{\frac{1}{2}}e^{-\eta^2(s^2-\tau^2)L^s}(L^s)^{-\frac{1}{2}}\partial_x(\partial_s \AP(x,s)\partial_x w(x,\eta^2\tau^2,s))\Vert_{L^2(\mathbb{R}^n)} d\tau\Big)^{2} ds
        \\
        &\qquad\lesssim \eta^2\int_0^{3l(\Delta)}\frac{1}{s}\Big(\int_0^{s}\tau \Vert \partial_x e^{-\eta^2(s^2-\tau^2)L^s}(L^s)^{-\frac{1}{2}}\partial_x(\partial_s \AP(x,s)\partial_x w(x,\eta^2\tau^2,s))\Vert_{L^2(\mathbb{R}^n)} d\tau\Big)^{2} ds
        \\
        &\qquad\lesssim \eta^2\int_0^{3l(\Delta)}\frac{1}{s}\Big(\int_0^{s}\frac{\tau}{\sqrt{s^2-\tau^2}}\Vert(L^s)^{-\frac{1}{2}}\partial_x(\partial_s \AP(x,s)\partial_x w(x,\eta^2\tau^2,s))\Vert_{L^2(\mathbb{R}^n)} d\tau\Big)^2 ds
        \\
        &\qquad\lesssim \eta^2\int_0^{3l(\Delta)}\frac{1}{s}\Big(\int_0^{s}\frac{\tau}{\sqrt{s^2-\tau^2}}\Big(\int_{\mathbb{R}^n} |\partial_s A(x,s)\partial_x w(x,\eta^2\tau^2,s))|^2 dx\Big)^{1/2} d\tau\Big)^2 ds
    \end{align*}
    
    Let us now use \refprop{prop:off-diagonal} for \(w(x,\eta^2 t^2,s)=e^{-\eta^2t^2L^s}\phi^s\). We have
    \begin{align*}
        &\Big(\int_{\mathbb{R}^n} |\partial_s A(x,s)\partial_x e^{-\eta^2\tau^2L^s}\phi^s(x)|^2 dx\Big)^{1/2}
        \\
        &=\sum_{k=1}^\infty \Big(\int_{2^k 3\Delta \setminus 2^{k-1}3\Delta}|\partial_s A(x,s)\partial_x e^{-\eta^2\tau^2L^s}\phi^s(x)|^2 dx\Big)^{1/2} + \Big(\int_{3\Delta} |\partial_s A(x,s)\partial_x e^{-\eta^2\tau^2L^s}\phi^s(x)|^2 dx\Big)^{1/2}
        \\
        &=\sum_{k=1}^\infty \frac{1}{s}e^{-c\frac{2^{k} l(\Delta)^2}{\tau^2}}\Vert\phi^s\Vert_{L^2} + \Big(\sum_{\substack{W\in\mathcal{D}(5\Delta) \\ l(W)\approx \tau}} \sup_{x\in W}|\partial_s A(x,s)|^2\int_W|\partial_x e^{-\eta^2\tau^2L^s}\phi^s(x)|^2 dx\Big)^{1/2},
    \end{align*}
    where we made use of the pointwise bound of \(|\partial_s A|\lesssim \frac{1}{s}\). By Poincar\'{e} inequality, we obtain \(\Vert\phi^s\Vert_{L^2}\lesssim l(\Delta)\Vert\partial_x \phi^s\Vert_{L^2}\lesssim l(\Delta)|\Delta|^{1/2}\), whence
    
    \begin{align*}
        \sum_{k=1}^\infty \frac{l(\Delta)}{s}e^{-c\frac{2^{k} l(\Delta)^2}{\tau}}|\Delta|^{1/2}\leq \sum_{k=1}^\infty \frac{l(\Delta)}{\tau}e^{-c\frac{2^{k} l(\Delta)^2}{\tau^2}}|\Delta|^{1/2}\lesssim \frac{\tau}{l(\Delta)}|\Delta|^{1/2}\lesssim |\Delta|^{1/2}.
    \end{align*}

    For the second term we apply \reflemma{lemma:thetaPointwiseBoundOnGoodSetF} and obtain
    \begin{align*}
        &\Big(\sum_{\substack{W\in\mathcal{D}(5\Delta) \\ l(W)\approx \tau}} \sup_{x\in W}|\partial_s A(x,s)|^2\int_W|\partial_x e^{-\eta^2\tau^2L^s}\phi^s(x)|^2 dx\Big)^{1/2}
        \\
        &\qquad\lesssim\Big(\sum_{\substack{W\in\mathcal{D}(5\Delta) \\ l(W)\approx \tau}} \sup_{x\in W}|\partial_s A(x,s)|^2\int_W|M[\partial_x\phi^s]|^2 dx\Big)^{1/2}
        \\
        &\qquad\lesssim\Big(\int_{5\Delta}\sup_{x\in B(x,s,s/2)}|\partial_s A(x,s)|^2 dx\Big)^{1/2}.
    \end{align*}

    Hence in total we get
    \begin{align*}
        &\int_{T(3\Delta)}\frac{|w_s^{(2)}|^2}{s}dxds
        \\
        &\lesssim \eta^2\int_0^{3l(\Delta)}\frac{1}{s}\Big(\int_0^{s}\frac{\tau}{\sqrt{s^2-\tau^2}}\Big(|\Delta|^{1/2} + \Big( \int_{5\Delta}\sup_{x\in B(x,s,s/2)}|\partial_s A(x,s)|^2 dx\Big)^{1/2}\Big) d\tau\Big)^2 ds
        \\
        &\lesssim \eta^2\int_0^{3l(\Delta)}\frac{1}{s}\Big(|\Delta| +  \int_{5\Delta}\sup_{x\in B(x,s,s/2)}|\partial_s A(x,s)|^2 dx\Big)\Big(\int_0^{s}\frac{\tau}{\sqrt{s^2-\tau^2}} d\tau\Big)^2 ds
        \\
        &\lesssim \eta^2\int_0^{3l(\Delta)}s\Big(|\Delta| +  \int_{5\Delta}\sup_{x\in B(x,s,s/2)}|\partial_s A(x,s)|^2 dx\Big) ds
        \\
        &\lesssim \eta^2l(\Delta)^2|\Delta| + \int_{T(5\Delta)}\sup_{x\in B(x,s,s/2)}|\partial_s A(x,s)|^2s dx ds
        \lesssim |\Delta|.
    \end{align*}

\end{proof}

\begin{lemma}\label{lemma:SqFctBoundsForw_s^2}
Let \(Q\subset \partial\Omega\) be a boundary cube of size \(s\), i.e. \(l(Q)\approx s\), \((t_j)_{j\in\Z}=((c2)^{j})_{j\in\Z}\) be a partition of \((0,\infty)\) for a constant \(\frac{1}{2}<c<1\), and the index \(i\) such that \(s\in (t_i,t_{i+1})\). We have the following local bound involving \(w_s^{(2)}\)
\begin{align}
    \int_Q |\partial_x w_s^{(2)}(x,s)|^2dx&\lesssim \frac{1}{s^2}\fint_{t_{l-1}}^{t_{l+2}}\int_{2Q}|v(x,k,s)|^2dkdx + \Vert\partial_s A\Vert_{L^\infty(2Q\times[t_{l-3},t_{l+4}])}^2|Q|\kappa_0^2.\label{lemma:localEstimateNablaw_s^2}
\end{align}
As a consequence we have the square function bound
\begin{align}
   \int_{T(\Delta)}|\partial_x w_s^{(2)}|^2sds\lesssim |\Delta|\label{lemma:w_s^2SqFctBound2}
\end{align}
and
\begin{align}
   \int_{T(\Delta)}|\partial_x w_s^{(2)}|ds\lesssim |\Delta|.\label{lemma:w_s^2SqFctBound3}
\end{align}

\end{lemma}

\begin{proof}

To start with note that from \cite{ulmer_mixed_2024} we already know that
\begin{align}
    \int_Q |\partial_x w_s^{(2)}(x,s)|^2dx&\lesssim \frac{1}{s^2}\fint_{t_{l-1}}^{t_{l+2}}\int_{2Q}|v(x,k,s)|^2dkdx\nonumber
    \\
    &\qquad+ \Vert\partial_s A(\cdot,s)\Vert_{L^\infty(2Q\times(t_{l-3},t_{l+4}))}^2\fint_{t_{l-3}}^{t_{l+4}}\int_{2Q} |\partial_x e^{-t^2L^s}\phi_s|^2dxdt \nonumber
    \\
    &\qquad+ \Vert\partial_s A(\cdot,s)\Vert_{L^\infty(2Q\times(t_{l-3},t_{l+4}))}^2 s^2 \fint_{t_{l-3}}^{t_{l+4}}\int_{2Q}|\partial_x\partial_ke^{-k^2L^s}\phi_s|^2dxdk.
\end{align}
By \reflemma{lemma:thetaPointwiseBoundOnGoodSetF} we obtain that 
\[\fint_{t_{i-3}}^{t_{i+4}}\int_{2Q} |\partial_x e^{-k^2L^s}\phi_s|^2 + s^2|\partial_x\partial_ke^{-k^2L^s}\phi_s|^2dxdk\lesssim |Q|\inf_{x\in Q} M[\Phi](x)\lesssim \kappa_0|Q|,\]
whence \eqref{lemma:localEstimateNablaw_s^2} follows. 

\medskip
To show \eqref{lemma:w_s^2SqFctBound2} we note that
\begin{align*}
    \int_{3\Delta}|v(x,k,s)|^2dx&\leq\Big(\int_0^k2\tau \Vert e^{-(k^2-\tau^2)L^s}\partial_x(\partial_s\AP\partial_x e^{-\tau^2L^s}\phi^s)\Vert_{L^2} d\tau\Big)^2
    \\
    &\leq\Big(\int_0^k \frac{2\tau}{\sqrt{k^2-\tau^2}} \Vert\partial_s\AP\partial_x e^{-\tau^2L^s}\phi^s\Vert_{L^2} d\tau\Big)^2.
\end{align*}
Observing that \reflemma{lemma:thetaPointwiseBoundOnGoodSetF}, \(L^2-L^2\) off-diagonal estimates in \refprop{prop:off-diagonal}, and Poincar\'{e} yield
\begin{align*}
    \Vert\partial_s\AP\partial_x e^{-\tau^2L^s}\phi^s\Vert_{L^2}&\lesssim\Big(\sum_{\substack{Q\subset 5\Delta\\Q\textrm{ finite overlap}, l(Q)\approx s\approx k}}\sup_Q|\partial_s\AP|^2|Q| \kappa_0^2\Big)^{1/2}
    \\
    &\hspace{30mm}+ \frac{1}{s}\Vert\partial_x e^{-\tau^2L^s}\phi^s\Vert_{L^2(\mathbb{R}^n\setminus 5\Delta)}
    \\
    &\lesssim\Big(\int_{5\Delta} \sup_{B(x,s,s/2)}|\partial_s\AP|^2dx\Big)^{1/2} + \frac{1}{s}\frac{1}{\tau}e^{-c\frac{l(\Delta)^2}{\tau^2}}\Vert\phi^s\Vert_{L^2(\mathbb{R}^n\setminus 5\Delta)}
    \\
    &\lesssim\Big(\int_{5\Delta} \sup_{B(x,s,s/2)}|\partial_s\AP|^2dx\Big)^{1/2} + \frac{1}{l(\Delta)^2}\Vert\phi^s\Vert_{L^2(\mathbb{R}^n\setminus 5\Delta)}
    \\
    &\lesssim\Big(\int_{5\Delta} \sup_{B(x,s,s/2)}|\partial_s\AP|^2dx\Big)^{1/2} + \frac{1}{\sqrt{l(\Delta)}},
\end{align*}
gives 
\begin{align*}
    \int_{3\Delta}|v(x,k,s)|^2dx&\leq s^2\int_{5\Delta} \sup_{B(x,s,s/2)}|\partial_s\AP|^2dx + \frac{s^4}{l(\Delta)} .
\end{align*}
Hence in total, we obtain
\begin{align*}
    \int_{T(3\Delta)} |\partial_x w_s^{(2)}(x,s)|^2 sdxds&\lesssim \int_{T(3\Delta)}2\sup_{B(x,s,s/2)}|\partial_s\AP|^2sdxds + \int_0^{l(\Delta)}\frac{s^3}{l(\Delta)}ds
    \\
    &\lesssim |\Delta|,
\end{align*}
which completes the proof of \eqref{lemma:w_s^2SqFctBound2}.

\medskip

Lastly for \eqref{lemma:w_s^2SqFctBound3}, we first fix \(s>0\) and let \(Q_j=Q_j^s\subset 5\Delta\) such that \(l(Q_j)\approx s\), the \(Q_j\) and \(Q_{j+1}\) have nontrivial intersection but all the other pairs \(Q_j\) and \(Q_i\) do not intersect, and the union of the \(Q_j\) covers \(3\Delta\). Then we observe for a dualising functions \(h_i\in L^2(Q_i)\) and \(Q_i\subset 3\Delta\)

\begin{align*}
    \Big(\int_{Q_i}|v(x,k,s)|^2dx\Big)^{1/2}&\leq \int_0^k2\tau \Vert e^{-(k^2-\tau^2)L^s}\partial_x(\partial_s\AP\partial_x e^{-\tau^2L^s}\phi^s)\Vert_{L^2(Q)} d\tau
    \\
    &\leq\int_0^k\tau \int_{\mathbb{R}^n}\partial_s\AP\partial_x e^{-\tau^2L^s}\phi^s  \partial_x e^{-(k^2-\tau^2)L^s} h_i dx d\tau
    \\
    &\leq\int_0^k\tau\Big( \int_{4\Delta}\partial_s\AP\partial_x e^{-\tau^2L^s}\phi^s  \partial_x e^{-(k^2-\tau^2)L^s} h_i dx
    \\
    &\qquad + \frac{1}{s}\sum_{j\geq 2} \int_{2^{j+1}\Delta\setminus 2^j\Delta}\partial_s\AP\partial_x e^{-\tau^2L^s}\phi^s  \partial_x e^{-(k^2-\tau^2)L^s} h_i dx\Big) d\tau.
\end{align*}
For the first term we have by \reflemma{lemma:thetaPointwiseBoundOnGoodSetF} and Cauchy Schwarz inequality and \(L^2-L^2\) off diagonal estimates (\refprop{prop:off-diagonal})
\begin{align*}
    &\int_{4\Delta}\partial_s\AP\partial_x e^{-\tau^2L^s}\phi^s  \partial_x e^{-(k^2-\tau^2)L^s} h_i dx
    \\
    &\leq\sum_j\int_{Q_j}\partial_s\AP\partial_x e^{-\tau^2L^s}\phi^s  \partial_x e^{-(k^2-\tau^2)L^s} h_i dx
    \\
    &\lesssim \sum_j \sup_{Q_j} |\partial_s\AP|\kappa_0\sqrt{|Q_j|}\frac{1}{\sqrt{k^2-\tau^2}}e^{-c\frac{|i-j|^2s^2}{k^2-\tau^2}}\Vert h_i\Vert_{L^2}
    \\
    &\lesssim \sum_j \sup_{Q_j} |\partial_s\AP|\kappa_0\frac{1}{\sqrt{s}}e^{-c|i-j|^2}\Vert h_i\Vert_{L^2}.
\end{align*}
We can also notice that summing over \(i\) yields
\begin{align*}
    &\sum_i \int_{4\Delta}\partial_s\AP\partial_x e^{-\tau^2L^s}\phi^s  \partial_x e^{-(k^2-\tau^2)L^s} h_i dx
    \\
    &\lesssim \sum_j \sup_{Q_j} |\partial_s\AP|\kappa_0\frac{1}{\sqrt{s}}\sum_i e^{-c|i-j|^2}\Vert h_i\Vert_{L^2}
    \\
    &\lesssim \sum_j \sup_{Q_j} |\partial_s\AP|\kappa_0\frac{1}{\sqrt{s}}.
\end{align*}

For the second term we have by \(L^2-L^2\) off diagonal estimates
\begin{align*}
    &\int_{2^{j+1}\Delta\setminus 2^j\Delta}\partial_s\AP\partial_x e^{-\tau^2L^s}\phi^s  \partial_x e^{-(k^2-\tau^2)L^s} h_i dx
    \\
    &\lesssim \frac{1}{s}\Vert \partial_x e^{-\tau^2L^s}\phi^s\Vert_{L^2(2^{j+1}\Delta\setminus 2^j\Delta)}\Vert\partial_x e^{-(k^2-\tau^2)L^s} h_i\Vert_{L^2(2^{j+1}\Delta\setminus 2^j\Delta)}
    \\
    &\lesssim \frac{1}{s}2^{(j-1)/2}l(\Delta)^{1/2}\kappa_0\frac{1}{\sqrt{k^2-\tau^2}}e^{-c\frac{2^{2j}l(\Delta)^2}{k^2-\tau^2}}\Vert h_i\Vert_{L^2(Q)}.
\end{align*}
Summing over \(j\) yields
\begin{align*}
    &\sum_j\frac{1}{s}2^{(j-1)/2}l(\Delta)^{1/2}\kappa_0\frac{1}{\sqrt{k^2-\tau^2}}e^{-c\frac{2^{2j}l(\Delta)^2}{k^2-\tau^2}}\Vert h_i\Vert_{L^2(Q)}
    \\
    &\lesssim \frac{1}{s}l(\Delta)^{1/2}\kappa_0\frac{1}{\sqrt{k^2-\tau^2}}\frac{(k^2-\tau^2)^{3/4}}{l(\Delta)^{3/2}}
    \\
    &\lesssim \frac{1}{\sqrt{s}l(\Delta)}. 
\end{align*}
In total this concludes the proof of \eqref{lemma:w_s^2SqFctBound3} by
\begin{align*}
    \int_{T(\Delta)}|\partial_x w_s^{(2)}| dxds&\lesssim \int_0^{l(\Delta)}\sum_{Q_i=Q_i^{s}}\Big(\int_{Q_i}|\partial_x w_s^{(2)}|^2dx\Big)^{\frac{1}{2}}|Q_i|^\frac{1}{2} ds
    \\
    &\lesssim\int_0^{l(\Delta)}\sum_{Q_i=Q_i^{s}}\big(\frac{1}{s}\fint_{t_l}^{t_{l+1}}\int_{Q_i} |v|^2dxdk + \sup_{B(x,s,s/2)}|\partial_s\AP||Q_i|^{1/2}\big)|Q_i|^\frac{1}{2} ds
    \\
    &\lesssim\int_0^{l(\Delta)}\frac{1}{s}\fint_{t_l}^{t_{l+1}}\int_0^k\tau \sum_{Q_j=Q_j^s} \big(\sup_{Q_j} |\partial_s\AP|\kappa_0
    +\frac{1}{l(\Delta)}\big)d\tau ds \\
    &\qquad + \int_0^{l(\Delta)}\sum_{Q_i=Q_i^{s}} \sup_{B(x,s,s/2)}|\partial_s\AP||Q_i| ds
    \\
    &\lesssim\int_0^{l(\Delta)}\sum_{Q_j=Q_j^{s}} \sup_{Q_j} |\partial_s\AP| s ds + \int_0^{l(\Delta)}\frac{l(\Delta)}{s}\frac{s}{l(\Delta)}ds
    \\
    &\qquad + \int_{T(3\Delta)}\sup_{B(x,s,s/2)}|\partial_s\AP| dxds
    \\
    &\lesssim\int_{T(3\Delta)}\sup_{B(x,s,s/2)}|\partial_s\AP| dx ds + l(\Delta).
\end{align*}

\end{proof}

\subsection{The partial derivative part \(w_t\)}

\begin{lemma}\label{lemma:SqFctBoundsForw_t}
The following square function bounds involving \(w_t\) hold:
\begin{enumerate}[(i)]
    \item \[\int_{T(\Delta)}\frac{|w_t|^2}{s}dxds=\int_{T(\Delta)}\frac{|sL^se^{-\eta^2s^2L^s}\phi^s|^2}{s}dxds\lesssim |\Delta|;\]\label{lemma:w_tSqFctBound1}
    \item \[\int_{T(\Delta)}\frac{|s\partial_x w_t|^2}{s}dxds\lesssim |\Delta|;\label{lemma:w_tSqFctBound2}\]
    \item \[\int_{T(\Delta)}\frac{|s^2L^sw_t|^2}{s}dxds\lesssim |\Delta|.\label{lemma:w_tSqFctBound3}\]
\end{enumerate}
\end{lemma}

\begin{proof}
We start with proving \eqref{lemma:w_tSqFctBound1}. First, note that we can write \(\AP=B^2\) for \(B>0\). And hence we obtain
\begin{align*}
    \int_{T(\Delta)}\frac{|sL^se^{-\eta^2s^2L^s}\phi^s|^2}{s}dxds&\leq\int_{\mathbb{R}^n\times(0,l(\Delta))}\frac{|sL^se^{-\eta^2s^2L^s}\phi^s|^2}{s}dxds
    \\
    &=\int_{\mathbb{R}^n\times (0,l(\Delta))} B\partial_x L^{s}e^{-\eta^2tL^s}\phi^s \cdot B\partial_x e^{-\eta^2s^2L^s}\phi^s dxdt
    \\
    &=\int_{\mathbb{R}^n\times (0,l(\Delta))} B \partial_s L^se^{-\eta^2s^2L^s}\phi^s \cdot B\partial_x e^{-\eta^2s^2L^s}\phi^s dxds
    \\
    &\qquad - \int_{\mathbb{R}^n\times (0,l(\Delta))} B \partial_x w_s^{(1)} \cdot B\partial_x e^{-\eta^2s^2L^s}\phi^s dxds
    \\
    &\qquad - \int_{\mathbb{R}^n\times (0,l(\Delta))} B \partial_x w_s^{(2)} \cdot B\partial_x e^{-\eta^2s^2L^s}\phi^s dxds 
    \\
    &=\int_{0} ^{l(\Delta)}\partial_s\Vert B\partial_x e^{-\eta^2s^2L^s}\phi^s\Vert_{L^2(\mathbb{R}^n)}^2 dxds
    \\
    &\qquad + \int_{\mathbb{R}^n\times (0,l(\Delta))} \partial_s\AP \partial_x e^{-\eta^2s^2L^s}\phi^s \cdot \partial_x e^{-\eta^2s^2L^s}\phi^s dxds
    \\
    &\qquad - \int_{\mathbb{R}^n\times (0,l(\Delta))} B \partial_x w_s^{(1)} \cdot B\partial_x e^{-\eta^2s^2L^s}\phi^s dxds
    \\
    &\qquad - \int_{\mathbb{R}^n\times (0,l(\Delta))} B \partial_x w_s^{(2)} \cdot B\partial_x e^{-\eta^2s^2L^s}\phi^s dxds 
    \\
    &=:I+II+III+IV.
\end{align*}

Looking at the integrals separately, first we have
\[I=\Vert B\partial_x e^{-\eta^2(3l(\Delta))^2L^s}\phi^{3l(\Delta)}\Vert_{L^2(\mathbb{R}^n)}^2 - \Vert B\partial_x \phi^0\Vert_{L^2(\mathbb{R}^n)}^2 \lesssim |\Delta|,\]
and
\[II\lesssim \Vert \sup |\partial_s \AP|\Vert_{\mathcal{C}}\int_{\mathbb{R}^n} \tilde{N}_2(|\partial_x e^{-\eta^2s^2L^s}\phi^s|)^2 dx. \]
From \reflemma{lemma:UniformBoundOnM[nablaphi^s]} we have
\[\int_{5\Delta}\tilde{N}_2(|\partial_x e^{-\eta^2s^2L^s}\phi^s|)^2 dx\lesssim |\Delta|,\]
and using \(L^2-L^2\) off diagonal estimates and Poincar\'{e}'s inequality, we also have
\begin{align*}
    \int_{\mathbb{R}^n\setminus 4\Delta}\tilde{N}_2(|\partial_x e^{-\eta^2s^2L^s}\phi^s|)^2 dx
    &\lesssim \sum_{j\geq 3}\int_{2^j \Delta\setminus 2^{j-1}\Delta}\tilde{N}_2(|\partial_x e^{-\eta^2s^2L^s}\phi^s|)^2 dx
    \\
    & \lesssim \sum_{j\geq 3}\int_{2^j \Delta\setminus 2^{j-1}\Delta}\frac{1}{2^{2j}}\frac{1}{l(\Delta)^2}\sup_s\Vert \phi^s\Vert_{L^2}^2 dx
    \\
    & \lesssim \sum_{j\geq 3}\int_{2^j \Delta\setminus 2^{j-1}\Delta}\frac{1}{2^{2j}} dx
    \\
    &\lesssim |\Delta|.
\end{align*}
Next, we have by \reflemma{lemma:thetaPointwiseBoundOnGoodSetF} 
\begin{align*}
    III&= \int_0^{l(\Delta)}\int_{\mathbb{R}^n}w_s^{(1)}(x,s)L^se^{-\eta^2s^2L^s}\phi^s dx ds
    \\
    &= \int_0^{l(\Delta)}\int_{\mathbb{R}^n}A\partial_x\partial_s\phi^s(x,s)\partial_xe^{-2\eta^2s^2L^s}\phi^s dx ds
    \\
    &\leq\int_0^{l(\Delta)}\sum_{\substack{Q\subset 5\Delta\\l(Q)\approx s, Q \textrm{ finite overlap}}} \sup_{y\in Q} \big(|\partial_s\AP(y,s)| + \fint_{3\Delta}|\partial_s A|dx\big)\big(\int_Q |\partial_x e^{-2\eta^2s^2L^s}\phi^s|^2\big)^{1/2}ds
    \\
    &\leq\int_{T(5\Delta)}\sup_{(y,t)\in B(x,s,s/2)}|\partial_s A(y,t)| dxds\lesssim |\Delta|.
\end{align*}

At last, we have
\begin{align*}
    |IV|&=\int_0^{l(\Delta)}\int_0^s\int_{\mathbb{R}^n}|2\tau\partial_s\AP\partial_x e^{-\eta^2\tau^2L^s}\phi^s\cdot \partial_x e^{-\eta^2(s^2-\tau^2)L^s}L^se^{-\eta^2 s^2L^s}\phi^s| dx dt ds
    \\
    &\leq \Vert \sup \partial_s A\Vert_{\mathcal{C}}\int_{\mathbb{R}^n}\tilde{N}\Big(\int_0^s \tau|\partial_x e^{-\eta^2\tau^2L^s}\phi^s| |\partial_x e^{-\eta^2(s^2-\tau^2)L^s}L^se^{-\eta^2 s^2L^s}\phi^s|d\tau\Big)(x) dx
\end{align*}

We note that the appearing nontangential maximal function is pointwise bounded. To see this we take \(x\in 5\Delta, t>0\) and have
\begin{align*}
    &\fint_{B(x,t,t/2)} \int_0^s \tau|\partial_x e^{-\eta^2\tau^2L^s}\phi^s| |\partial_x e^{-\eta^2(s^2-\tau^2)L^s}L^se^{-\eta^2 s^2L^s}\phi^s|d\tau dxds
    \\
    &\lesssim \fint_{t/2}^{3t/2} \int_0^s \tau\Big(\fint_{B(x,t/2)} |\partial_x e^{-\eta^2\tau^2L^s}\phi^s|^2dx\Big)^{1/2} \Big(\fint_{B(x,t/2)}|\partial_x e^{-\eta^2(s^2-\tau^2)L^s}L^se^{-\eta^2 s^2L^s}\phi^s|^2 dx\Big)^{1/2} d\tau ds
\end{align*}  

By \(L^2-L^2\)-off diagonal estimates (\refprop{prop:off-diagonal}), \reflemma{lemma:thetaPointwiseBoundOnGoodSetF}, and \refcor{cor:PointwiseBoundsofw_tAndrhoByM(nablaphi^s)} we obtain 
\[\Big(\fint_{B(x,t/2)} |\partial_x e^{-\eta^2\tau^2L^s}\phi^s|^2dx\Big)^{1/2}\lesssim \frac{1}{t}\frac{t}{\tau}\tau \kappa_0\lesssim \kappa_0,\]
and
\begin{align*}
    \Big(\fint_{B(x,t/2)}|\partial_x e^{-\eta^2(s^2-\tau^2)L^s}L^se^{-\eta^2 s^2L^s}\phi^s|^2 dx\Big)^{1/2}&\lesssim \frac{1}{\sqrt{s^2-\tau^2}}M[|L^se^{-\eta^2 s^2L^s}\phi^s|^2]^{1/2}
    \\
    &\lesssim \frac{1}{\sqrt{s^2-\tau^2}}\frac{1}{s}M[|w_t(\cdot,s)|^2]^{1/2}
    \\
    &\lesssim \frac{1}{\sqrt{s^2-\tau^2}}\frac{1}{s}M[M[|\partial_x\phi^s|]^2]^{1/2}.
\end{align*}
Hence in total with \reflemma{lemma:UniformBoundOnM[nablaphi^s]}
\begin{align*}\fint_{B(x,t,t/2)} &\int_0^s \tau|\partial_x e^{-\eta^2\tau^2L^s}\phi^s| |\partial_x e^{-\eta^2(s^2-\tau^2)L^s}L^se^{-\eta^2 s^2L^s}\phi^s|d\tau dxds
\\
&\qquad\lesssim  \fint_{t/2}^{3t/2}\int_0^s\tau \frac{1}{\sqrt{s^2-\tau^2}}\frac{1}{s}d\tau ds\leq C. \end{align*}

and

\begin{align*}
    &\int_{5\Delta}\tilde{N}\Big(\int_0^s \tau|\partial_x e^{-\eta^2\tau^2L^s}\phi^s| |\partial_x e^{-\eta^2(s^2-\tau^2)L^s}L^se^{-\eta^2 s^2L^s}\phi^s|d\tau\Big)(x) dx
    \\
    &\lesssim |\Delta|
\end{align*}

For \(x\in 2^j \Delta\setminus 2^{j-1}\Delta\mathbb{R}^n\setminus 4\Delta, 3l(\Delta)>t>0\) and \(j\geq 3\) however, we obtain with \(L^2-L^2\) off diagonal estimates
\begin{align*}
    \Big(\fint_{B(x,t/2)} |\partial_x e^{-\eta^2\tau^2L^s}\phi^s|^2dx\Big)^{1/2}&\lesssim \frac{1}{\sqrt{t}}\frac{1}{\tau}e^{-c\frac{2^{2j}l(\Delta)^2}{\tau^2}}\Vert \phi^s\Vert_{L^2(3\Delta)}\lesssim  \frac{1}{\sqrt{t}}\frac{1}{2^{j}l(\Delta)}\Vert \phi^s\Vert_{L^2(3\Delta)}
    \\
    &\lesssim  \frac{1}{2^j\sqrt{t}}\Vert \partial_x\phi^s\Vert_{L^2(3\Delta)}\lesssim  \frac{1}{2^j\sqrt{t}}\sqrt{|\Delta|},
\end{align*}
and
\begin{align*}
    \Big(\fint_{B(x,t/2)}|\partial_x e^{-\eta^2(s^2-\tau^2)L^s}L^se^{-\eta^2 s^2L^s}\phi^s|^2 dx\Big)^{1/2}&\lesssim \frac{1}{\sqrt{t}}\frac{1}{\sqrt{s^2-\tau^2}}e^{-c\frac{2^{2j}l(\Delta)^2}{s^2-\tau^2}}\Vert L^se^{-\eta^2 s^2L^s}\phi^s\Vert_{L^2}
    \\
    &\lesssim \frac{1}{\sqrt{t}}\frac{1}{2^jl(\Delta)}\frac{1}{s}\Vert M[\partial_x\phi^s]\Vert_{L^2(3\Delta)}
    \\
    &\lesssim \frac{1}{2^{j}\sqrt{t}}\frac{1}{s}\frac{1}{\sqrt{|\Delta|}}\kappa_0
\end{align*}

Hence we obtain for \(x\in 2^j \Delta\setminus 2^{j-1}\Delta\)
\begin{align*}\fint_{B(x,t,t/2)} &\int_0^s \tau|\partial_x e^{-\eta^2\tau^2L^s}\phi^s| |\partial_x e^{-\eta^2(s^2-\tau^2)L^s}L^se^{-\eta^2 s^2L^s}\phi^s|d\tau dxds
\\
&\qquad\lesssim  \fint_{t/2}^{3t/2}\int_0^s\tau \frac{1}{2^j\sqrt{t}}\sqrt{|\Delta|}\frac{1}{2^{j}\sqrt{t}}\frac{1}{s}\frac{1}{\sqrt{|\Delta|}}d\tau ds\leq \frac{1}{2^{2j}}, \end{align*}
and

\begin{align*}
    &\int_{\mathbb{R}^n\setminus 5\Delta}\tilde{N}\Big(\int_0^s \tau|\partial_x e^{-\eta^2\tau^2L^s}\phi^s| |\partial_x e^{-\eta^2(s^2-\tau^2)L^s}L^se^{-\eta^2 s^2L^s}\phi^s|d\tau\Big)(x) dx
    \\
    &\lesssim \sum_{j\geq 3}\int_{2^j\Delta\setminus 2^{j-1}\Delta}\tilde{N}\Big(\int_0^s \tau|\partial_x e^{-\eta^2\tau^2L^s}\phi^s| |\partial_x e^{-\eta^2(s^2-\tau^2)L^s}L^se^{-\eta^2 s^2L^s}\phi^s|d\tau\Big)(x) dx
    \\
    &\lesssim  \sum_{j\geq 3}\int_{2^j\Delta\setminus 2^{j-1}\Delta}\frac{1}{2^{2j}}dx\lesssim |\Delta|,
\end{align*}
whence the integral \(|IV|\lesssim |\Delta|\).

\hfill\\
The proofs of \eqref{lemma:w_tSqFctBound2} and \eqref{lemma:w_tSqFctBound3} rely on \eqref{lemma:w_tSqFctBound1} and \refprop{prop:L2NormBoundsOfHeatSemigroup}. For \eqref{lemma:w_tSqFctBound2} we see that
\begin{align*}
    &\int_0^{l(\Delta)}\int_{\Delta}\frac{|s^2\partial_x L^s e^{-\eta^2s^2L^s}\phi^s|^2}{s}dxds
    \\
    &\leq\int_0^{l(\Delta)}\int_{\mathbb{R}^n}\frac{|s^2\partial_x L^s e^{-\eta^2s^2L^s}\phi^s|^2}{s} dxds
    \\
    &=\int_0^{l(\Delta)}\int_{\mathbb{R}^n}\frac{|s\partial_x e^{-\frac{\eta^2s^2}{2}L^s}(sL^s e^{-\frac{\eta^2s^2}{2}L^s}\phi^s)|^2}{s} dxds
    \\
    &\lesssim \int_0^{l(\Delta)}\int_{\mathbb{R}^n}\frac{|s L^s e^{-\frac{\eta^2s^2}{2}L^s}\phi^s|^2}{s} dxds,
\end{align*}
where the last integral is bounded by \eqref{lemma:w_tSqFctBound1} after a change of variable argument. Similar we observe for \eqref{lemma:w_tSqFctBound3}
\begin{align*}
    &\int_{T(\Delta)}\frac{|s^2L^s L^s e^{-\eta^2s^2L^s}\phi^s|^2}{s}dxds
    \\
    &=\int_0^{l(\Delta)}\int_{\mathbb{R}^n}\frac{|s^2 L^s L^s e^{-\eta^2s^2L^s}\phi^s|^2}{s} dxds
    \\
    &=\int_0^{l(\Delta)}\int_{\mathbb{R}^n}\frac{|s L^s e^{-\frac{\eta^2s^2}{2}L^s}(sL^s e^{-\frac{\eta^2s^2}{2}L^s}\phi^s)|^2}{s} dxds
    \\
    &\lesssim \int_0^{l(\Delta)}\int_{\mathbb{R}^n}\frac{|s L^s e^{-\frac{\eta^2s^2}{2}L^s}\phi^s|^2}{s} dxds.
\end{align*}

\end{proof}

In contrast to \(\partial_s\phi^s-w_s^{(1)}\) and \(w_s^{(2)}\) we have a certain local Harnack-type property for \(w_t\). First we have 

\begin{lemma}\label{lemma:localHarnackTypeInequalityForw_t}
    Let \(Q\subset\mathbb{R}^n\) be a cube with side length \(l(Q)\approx s\approx R_0\) and \(\hat{Q}:=(1+\varepsilon)Q\) be an enlarged cube for some fixed \(\frac{1}{2}>\varepsilon>0\). Then
    \begin{align*}
    \sup_{(x,s)\in Q\times (R_0,2R_0]}w_t(x,s)
    &\lesssim \int_{(1-\varepsilon)R_0}^{2(1+\varepsilon)R_0}\fint_{\hat{Q}\times((1-\varepsilon)R_0,2(1+\varepsilon)R_0)}\frac{|w_t(x,s)|^2}{s}dxds 
    \\
    &\qquad + \int_{(1-\varepsilon)R_0}^{2(1+\varepsilon)R_0}\fint_{\hat{Q}\times((1-\varepsilon)R_0,2(1+\varepsilon)R_0)}|\partial_x v_1(x,t^2,s)|^2sdxdtds 
    \\
    &\qquad + \int_{(1-\varepsilon)R_0}^{2(1+\varepsilon)R_0}\fint_{\hat{Q}\times((1-\varepsilon)R_0,2(1+\varepsilon)R_0)}|\partial_x v_2(x,t^2,s)|^2sdxdtds.
\end{align*}

    As a consequence, we have
    \begin{align*}
    \sum_{k\leq k_0}\sum_{Q\in \mathcal{D}_k(\Delta)}|Q|\sup_{Q\times (2^{-k},2^{-k+1}]}|w_t|^2&\lesssim |\Delta|.
    \end{align*}
\end{lemma}

The consequence in above lemma resembles a version of \eqref{lemma:w_tSqFctBound1} in \reflemma{lemma:SqFctBoundsForw_t} where the supremum is taken on each Whitney cube. Hence it is not surprising that it holds, but the proof is technical.

\begin{proof}
    First, we would like to refer to the proof of Lemma 7.9 in \cite{ulmer_mixed_2024}, where an intermediary step proves
    \begin{align*}
    \sup_{(x,s)\in Q\times (R_0,2R_0]}w_t(x,s)
    &\lesssim \int_{(1-\varepsilon)R_0}^{2(1+\varepsilon)R_0}\fint_{\hat{Q}\times((1-\varepsilon)R_0,2(1+\varepsilon)R_0)}\frac{|w_t(x,s)|^2}{s}dxds 
    \\
    &\qquad + \int_{(1-\varepsilon)R_0}^{2(1+\varepsilon)R_0}\fint_{\hat{Q}\times((1-\varepsilon)R_0,2(1+\varepsilon)R_0)}|\partial_x v_1(x,t^2,s)|^2sdxdtds 
    \\
    &\qquad + \int_{(1-\varepsilon)R_0}^{2(1+\varepsilon)R_0}\fint_{\hat{Q}\times((1-\varepsilon)R_0,2(1+\varepsilon)R_0)}|\partial_x v_2(x,t^2,s)|^2sdxdtds.
\end{align*}
By \reflemma{lemma:UniformBoundOnM[nablaphi^s]} and \reflemma{lemma:thetaPointwiseBoundOnGoodSetF} we obtain
\begin{align*}
    &\int_{(1-\varepsilon)R_0}^{2(1+\varepsilon)R_0}\fint_{\hat{Q}\times((1-\varepsilon)R_0,2(1+\varepsilon)R_0)}|\partial_x v_1(x,t^2,s)|^2sdxdtds
    \\
    &\qquad\lesssim \int_{(1-\varepsilon)R_0}^{2(1+\varepsilon)R_0}\fint_{\hat{Q}\times((1-\varepsilon)R_0,2(1+\varepsilon)R_0)}|M[\partial_s\partial_x \phi^s]|^2sdxdtds
    \\
    &\qquad\lesssim \fint_{\hat{Q}\times ((1-\varepsilon)R_0,2(1+\varepsilon)R_0)}\sup_{\hat{Q}\times((1-\varepsilon)R_0,2(1+\varepsilon)R_0)}|\partial_s\AP|^2s dxds
    \\
    &\qquad\qquad + \fint_{\hat{Q}\times ((1-\varepsilon)R_0,2(1+\varepsilon)R_0)}\fint_\Delta |\partial_s \AP|^2sdx,
\end{align*}
whence
\begin{align*}
    &\sum_{k\leq k_0}\sum_{Q\in \mathcal{D}_k(\Delta)}\int_{\hat{Q}\times ((1-\varepsilon)R_0,2(1+\varepsilon)R_0)}\fint_{(1-\varepsilon)2^{-k}}^{(1+\varepsilon)2^{-k+1}}|\partial_x v_1(x,t^2,s)|^2sdtdxds
    \\
    &\qquad \lesssim \int_{5\Delta}\sup_{B(x,s,s/2)}|\partial_s\AP|^2sdxds\lesssim |\Delta|.
\end{align*}
Furthermore, we can refer to the proof of \eqref{lemma:w_s^2SqFctBound2} in \reflemma{lemma:SqFctBoundsForw_s^2} to see that
\begin{align*}
    \sum_{k\leq k_0}\sum_{Q\in \mathcal{D}_k(\Delta)}\int_{\hat{Q}\times ((1-\varepsilon)R_0,2(1+\varepsilon)R_0)}\fint_{(1-\varepsilon)2^{-k}}^{(1+\varepsilon)2^{-k+1}}|\partial_x v_2(x,t^2,s)|^2sdtdxds\lesssim |\Delta|.
\end{align*}

Hence 
\begin{align*}
    \sum_{k\leq k_0}\sum_{Q\in \mathcal{D}_k(\Delta)}|Q|\sup_{Q\times (2^{-k},2^{-k+1}]}|w_t|^2\lesssim \int_{3\Delta}\frac{|w_t(x,s)|^2}{s}dxds + |\Delta|
    \lesssim |\Delta|.
    \end{align*}
\end{proof}

\section{\(L^2\) estimates for square functions}

\begin{lemma}\label{lemma:L^2estimatesForSquareFunctions}
Assuming the \(L^1\) Carleson condition \eqref{condition:L^1Carlesontypecond} we have the following estimates:
\begin{enumerate}[(i)]
    \item \[\int_{T(\Delta)}|\nabla \partial_s \rho(x,s)|^2 sdxds\lesssim |\Delta|;\]
    \item \[\int_{T(\Delta)}|L^s \rho(x,s)|^2s dxds\lesssim |\Delta|;\]
    \item \[\int_{T(\Delta)}|\partial_s A(x,s)\nabla\rho(x,s)|^2s dxds\lesssim |\Delta|;\]
    \item \[\int_{T(\Delta)}\frac{|\partial_s \theta(x,s)|^2}{s}dxds\lesssim |\Delta|;\]
    \item \[\int_{T(\Delta)}\frac{|\theta(x,s)|^2}{s^3}dxds\lesssim |\Delta|.\]
\end{enumerate}
\end{lemma}

\begin{proof}

\begin{enumerate}[(i)]
    \item We split
    \begin{align*}
    \int_{T(\Delta)}|\nabla \partial_s \rho(x,s)|^2 sdxds&\leq \int_{T(\Delta)}|\nabla w_t(x,s)|^2 sdxds+\int_{T(\Delta)}|\nabla w_s^{(1)}(x,s)|^2 sdxds
    \\
    &\qquad+\int_{T(\Delta)}|\nabla w_s^{(2)}(x,s)|^2 sdxds,
    \end{align*}
    and apply \reflemma{lemma:SqFctBoundsForw_t}, \reflemma{lemma:SqFctBoundsForpartial_sPhiAndw_s^1} and \reflemma{lemma:SqFctBoundsForw_s^2} to get the required bound.
    
    \item We can observe that
    \[\int_{T(\Delta)}|L^s \rho(x,s)|^2s dxds=\int_{T(\Delta)}\frac{|w_t(x,s)|^2}{s} dxds\lesssim |\Delta|\]
    by \reflemma{lemma:SqFctBoundsForw_t}.

    \item We can calculate 
    \begin{align*}
    \int_{T(\Delta)}|\partial_s A(x,s)\nabla\rho(x,s)|^2s dxds&\leq \int_0^{l(\Delta)}\Vert\partial_s A(\cdot,s)\Vert_\infty\Big(\int_\Delta|\nabla\phi^s|^2+|\nabla e^{-\eta^2s^2L^s}\phi^s|^2dx\Big) s ds
    \\
    &\lesssim \int_0^{l(\Delta)}\Vert\partial_s A(\cdot,s)\Vert_\infty|\Delta|s ds \lesssim |\Delta|
    \end{align*}
    by \eqref{condition:L^2Carlesontypecond}.

    \item We split
    \begin{align*}
        \int_{T(\Delta)}\frac{|\partial_s \theta(x,s)|^2}{s}dxds&\leq \int_{T(\Delta)}\frac{|(\partial_s\phi^s-w_s^{(1)})(x,s)|^2}{s}dxds+\int_{T(\Delta)}\frac{|w_s^{(2)}(x,s)|^2}{s}dxds
        \\
        &\qquad+\int_{T(\Delta)}\frac{|w_t(x,s)|^2}{s}dxds
    \end{align*}
    and the bound follows from \reflemma{lemma:SqFctBoundsForpartial_sPhiAndw_s^1}, \reflemma{lemma:SqFctBoundsForw_s^2}, and \reflemma{lemma:SqFctBoundsForw_t}.

    \item The proof works analogously to the proof of Lemma 9 in \cite{hofmann_dirichlet_2022}.
\end{enumerate}
\end{proof}

%==================
%\input{Discoard}

\begin{comment}
\section{Conflict of interest statement}
On behalf of all authors, the corresponding author states that there is no conflict of interest.
\end{comment}

\bibliographystyle{alpha}
\bibliography{references} 
\end{document}